\documentclass[pdftex,twoside,a4paper,psamsfonts,fleqn,reqno]{amsart} 
\usepackage{mathdots}
\usepackage{amssymb}
\usepackage{graphicx} 
\usepackage{upref}
\usepackage[T1]{fontenc} 
\usepackage[english]{babel} 
\usepackage{color} 
\usepackage[bookmarks,
  bookmarksopen=true,
  bookmarksopenlevel=1,
  bookmarksdepth=2,
  pdftitle={On Local Models with Special Parahoric Level Structure},
  pdfauthor={Kai Arzdorf},
  pdfsubject={Diplomarbeit am Mathematischen Institut der Rheinischen Friedrich-Wilhelms-Universität Bonn},
  pdfcreator={LaTeX2e, hyperref},
  pdfkeywords={Diplomarbeit, Mathematics, Arithmetic Algebraic Geometry, Shimura Variety, Local Model, Special Parahoric Level Structure},
  pdfstartview=Fit,
  pdfdisplaydoctitle=true,
  pdfduplex=DuplexFlipLongEdge,
  colorlinks=true,
  linkcolor=blue,
  citecolor=colorDarkgreen,
  filecolor=red,
  urlcolor=colorDarkmagenta,
  menucolor=red,
  bookmarksnumbered=true
]{hyperref}
\usepackage[alphabetic,abbrev]{amsrefs} 

\definecolor{colorDarkgreen}{rgb}{0,0.5,0}
\definecolor{colorDarkred}{rgb}{0.8,0,0}
\definecolor{colorDarkmagenta}{rgb}{0.5,0,0.5}

\theoremstyle{plain}
\newtheorem{thm}{Theorem}[section]

\newtheorem{prp}[thm]{Proposition}
\newtheorem{cor}[thm]{Corollary}
\newtheorem{lem}[thm]{Lemma}

\theoremstyle{definition}

\newtheorem{rem}[thm]{Remark}

\theoremstyle{remark}

\newenvironment{prv}{\sffamily\slshape\small}{}
\newcommand{\bprv}{\begin{prv}}
\newcommand{\eprv}{\end{prv}}

\newcommand{\beqn}{\begin{equation}}
\newcommand{\eeqn}{\end{equation}}

\numberwithin{equation}{section}

\numberwithin{figure}{section}

\renewcommand{\dotfill}{\leaders\hbox to 5pt{\hss.\hss}\hfill}

\DeclareMathOperator{\id}{id}
\DeclareMathOperator{\tr}{Tr}
\DeclareMathOperator{\Grass}{Grass}
\DeclareMathOperator{\GU}{GU}
\DeclareMathOperator{\Sp}{Sp}
\DeclareMathOperator{\GL}{GL}
\DeclareMathOperator{\Spec}{Spec} 

\DeclareMathOperator{\Gal}{Gal}
\DeclareMathOperator{\Stab}{Stab}
\DeclareMathOperator{\diag}{diag}

\DeclareMathOperator{\spn}{span}
\DeclareMathOperator{\pr}{pr}
\newcommand{\T}{\text{{\normalfont t}}}

\newcommand{\naive}{\text{\normalfont naive}}
\newcommand{\loc}{\text{\normalfont loc}}
\newcommand{\red}{\text{\normalfont red}}

\newcommand{\sep}{\text{\normalfont sep}}
\newcommand{\lhs}{\text{\normalfont lhs}}
\newcommand{\rhs}{\text{\normalfont rhs}}

\newcommand{\twomat}[4]{\begin{pmatrix}#1&#2\\#3&#4\end{pmatrix}}
\newcommand{\threemat}[9]{\begin{pmatrix}#1&#2&#3\\#4&#5&#6\\#7&#8&#9\end{pmatrix}}
\newcommand{\fourtwomat}[8]{\begin{pmatrix}#1&#2\\#3&#4\\#5&#6\\#7&#8\end{pmatrix}}
\newcommand{\twovec}[2]{\begin{pmatrix}#1\\#2\end{pmatrix}}
\newcommand{\fourvec}[4]{\begin{pmatrix}#1\\#2\\#3\\#4\end{pmatrix}}
\newcommand{\twovex}[2]{\begin{pmatrix}#1&#2\end{pmatrix}}
\newcommand{\fourvex}[4]{\begin{pmatrix}#1&#2&#3&#4\end{pmatrix}}
\newcommand{\twomats}[4]{\bigl(\begin{smallmatrix}#1&#2\\#3&#4\end{smallmatrix}\bigr)}

\newcommand{\mat}[1]{{#1}}
\newcommand{\matx}[1]{{#1}}

\newcommand{\ot}{\otimes}
\newcommand{\iso}{\cong}

\newcommand{\I}{\mat{I}}
\newcommand{\Hh}{\mat{H}}
\newcommand{\J}{\mat{J}}
\newcommand{\Jp}{\mat{J'}}

\newcommand{\A}{\mathbb{A}}

\newcommand{\ZZ}{\mathbb{Z}}

\newcommand{\QQ}{\mathbb{Q}}

\newcommand{\G}{\mathcal{G}}
\newcommand{\F}{\mathcal{F}}
\newcommand{\OO}{\mathcal{O}}
\newcommand{\U}{\mathcal{U}}

\newcommand{\AAA}{\mathcal{A}}

\begin{document}
\selectlanguage{english} 


\pagenumbering{arabic}
\title[On Local Models with Special Parahoric Level Structure]{On Local Models with Special\\Parahoric Level Structure}
\author{Kai Arzdorf}
\address{Mathematisches Inst.\ der Universit\"at Bonn, Beringstr.~1, 53115 Bonn, Germany}
\email{\textcolor{colorDarkmagenta}{arzdorf@math.uni-bonn.de}} 
\date{March 2008}



\begin{abstract}
We consider the local model of a Shimura variety of PEL type, with the unitary similitudes corresponding to a ramified quadratic extension of $\QQ_p$ as defining group. We examine the cases where the level structure at $p$ is given by a parahoric that is the stabilizer of a selfdual periodic lattice chain and that is special in the sense of Bruhat--Tits theory. We prove that in these cases the special fiber of the local model is irreducible and generically reduced; consequently, the special fiber is reduced and is normal, Frobenius split, and with only rational singularities. In addition, we show that in these cases the local model contains an open subset that is isomorphic to affine space.
\end{abstract}

\maketitle


\tableofcontents


\nonfrenchspacing
\section*{Introduction}
\subsection*{Motivation and Main Results}
For the study of arithmetic properties of a variety over an algebraic number field, it is of interest to have a model over the ring of integers. In the particular case of a Shimura variety, one likes to have a model over the ring of integers $\OO_E$, where $E$ is the completion of the reflex field at a finite prime of residue characteristic $p$. It should be flat and have only mild singularities. If the Shimura variety is the moduli space over $\Spec E$ of abelian varieties with additional polarization, endomorphisms, and level structure (a Shimura variety of PEL type), it is natural to define a model by posing the moduli problem over $\OO_E$. In the case of a parahoric level structure at $p$ with the parahoric defined in an elementary way as the stabilizer of a selfdual periodic lattice chain, such a model has been given by Rapoport and Zink \cite{RZ:1996}.

Although in special cases this model is shown to be flat with reduced special fiber and with irreducible components that are normal and that have only rational singularities \cites{G:2001,G:2003}, in general it is not flat, as has been pointed out by Pappas \cite{P:2000}.
In a series of papers, Pappas and Rapoport \cites{PR:2003,PR:2005,PR:2007} examine how to define closed subschemes of this \emph{naive model} that are more likely to be flat. Flatness can be enforced by taking the flat closure of the generic fiber in the naive model. Aside from that, by adding further conditions one can attempt to cut out this closed subscheme, or at least give a better approximation. If the parahoric subgroup is the stabilizer of a selfdual periodic lattice chain, these questions can be reduced to problems of the corresponding \emph{local} models \cite{RZ:1996}. Locally for the \'etale topology around each point of the special fiber, these coincide with the corresponding moduli schemes. This approach has the advantage of leading to varieties that can be defined in terms of linear algebra and, thus, can be handled more easily. In this way, Pappas \cite{P:2000} defines the \emph{wedge local model}, a closed subscheme of the naive local model. The \emph{local model} is defined to be the closure of the generic fiber in the naive local model; it is also a closed subscheme of the wedge local model.

In one of their recent papers, Pappas and Rapoport \cite{PR:2007} study the case where the group defining the Shimura variety is the group of unitary similitudes corresponding to a quadratic extension of $\QQ$ that is ramified at $p$. Assuming the so-called Coherence Conjecture, the reducedness of the geometric special fiber of the local model is proven, and it is shown that its irreducible components are normal and with only rational singularities (loc.cit., Thm.~4.1). Some special cases, however, can be treated without relying on this conjecture. We will prove the following theorem:

\begin{thm}[main theorem, cf.\ Thm.~\ref{thm:specialparahoric:main}]\label{thm:introduction:main}
Let the level structure at $p$ be given by a parahoric that is defined in an elementary way as the stabilizer of a selfdual periodic lattice chain and that is \emph{special} in the sense of Bruhat--Tits theory \cite{T:1979}. Then the special fiber of the local model is irreducible and reduced; furthermore, the special fiber is normal, Frobenius split, and with only rational singularities.
\end{thm}

The proof of the theorem is divided into two major steps, in which we prove the following results:

\begin{thm}[first step, cf.\ Thm.~\ref{thm:reduced:reduced}]\label{thm:introduction:reduced}
Let the assumptions be the same as in the main theorem. Then the special fiber of the local model contains an open subset that is reduced.
\end{thm}

\begin{thm}[second step, cf.\ Thm.~\ref{thm:irreducibility:irreducible}]\label{thm:introduction:irreducible}
Under the assumptions of the main theorem, the special fiber of the local model is irreducible.
\end{thm}

Once it is shown that the special fiber of the local model is irreducible and generically reduced, the other properties stated in the main theorem follow by standard methods given in the paper by Pappas and Rapoport \cite{PR:2007}*{Proof of Thm.~5.1}.

It is shown in sect.\ 1.b.\ of loc.cit.\ that there are exactly three cases where the stabilizer subgroup is a special parahoric. Two of these cases have been treated by the authors of loc.cit.\ in sect.\ 5 of their paper, providing a proof of the theorems in these cases. The present paper is about the proof of the third case, which has not been treated (in full generality) in the literature yet, cf.\ Rem.~\ref{rem:specialparahoric:Rapoport}. Moreover, the results we obtain during the proof of Thm.~\ref{thm:introduction:reduced} are stronger than actually necessary:

\begin{thm}[cf.\ Thm.~\ref{thm:reduced:reduced} and Thm.~\ref{thm:other:reduced}]\label{thm:introduction:affine}
Let the same assumptions hold true as in the main theorem. Then the local model contains an open subset that is isomorphic to affine space.
\end{thm}

All of the above mentioned results are achieved by first evaluating the conditions of the wedge local model for open neighborhoods of certain special points (the ``best point'' and the ``worst point'', see sects.~\ref{ssec:reduced:bestpoint} and \ref{ssec:irreducibility:worstpoint}) and then passing to the actual local model using dimension arguments.

More precisely, the conditions of the wedge local model translate into several matrix identities, and we examine the schemes defined in this way. In the cases of Thm.~\ref{thm:introduction:reduced} and Thm.~\ref{thm:introduction:affine}, this leads to affine spaces described by simple matrix equations. In the case of Thm.~\ref{thm:introduction:irreducible}, we have to deal with a more complicated matrix scheme. We exploit that the symplectic group acts thereon, and by considering an equivariant projection morphism, we can confine ourselves to the study of certain fibers. These can be described following arguments by Pappas and Rapoport from their treatment of one of the other cases of a special parahoric level structure \cite{PR:2007}*{sect.~5.e}, using results of Ohta \cite{O:1986}*{Prop.\ 1 and Thm.\ 1} and of Kostant and Rallis \cite{KR:1971}*{Prop.~5 and its proof} on the structure of nilpotent orbits in the classical symmetric pair $(\mathfrak{gl}_n,\mathfrak{sp}_n)$.

By definition, the local model is flat; hence, its special fiber is equidimensional and has the same dimension as the generic fiber. The aforementioned matrix schemes are seen either to be irreducible of that dimension, or to contain a single irreducible component of that dimension with all other irreducible components having smaller dimension. Since the local model is a closed subscheme of the wedge local model, this allows transition to the local model.


\subsection*{Structure}
The paper is divided into five sections. In the first section we recall the construction of the local model for the situation considered above. In the second section we formulate the main theorem (Thm.~\ref{thm:introduction:main}) of this paper, with its two-part proof ranging over the following sections three and four, where we establish Thm.~\ref{thm:introduction:reduced} and Thm.~\ref{thm:introduction:irreducible}, respectively. As mentioned above, a slightly stronger result is obtained during the proof of Thm.~\ref{thm:introduction:reduced}. This carries over to the cases treated by Pappas and Rapoport and is the topic of the final section, cf.\ Thm.~\ref{thm:introduction:affine}.


\subsection*{Acknowledgments}
I conclude the introduction by thanking those people who helped and supported me in writing this paper. In particular, my thanks go to Prof.\ Dr.\ M.~Rapoport for introducing me to this fine area of mathematics and his steady interest in my work. I also thank Priv.-Doz.\ Dr.\ U.~G\"ortz for helping me with a multitude of questions. Finally, I am indebted to the Professor-Rhein-Stiftung for its financial support during my study.


\section{Definition of the Local Model}
We recall the construction of the local model for the general unitary group, as given in the recent paper by Pappas and Rapoport \cite{PR:2007}. We first introduce the basic notions and then define the naive local model. This is followed by a short discussion of the wedge local model, which provides a closed subscheme of the naive local model. Finally, we give the definition of the local model.


\subsection{Standard Lattices}
We use the notation of loc.cit. Let $F_0$ be a complete
discretely valued field with ring of integers $\OO_{F_0}$ and perfect residue field $k$ of characteristic $\neq 2$ and uniformizer $\pi_0$. Let $F/F_0$ be a ramified quadratic extension and $\pi\in F$ a uniformizer with $\pi^2=\pi_0$. Let $V$ be an $F$-vector space of dimension $n\ge3$ with an $F/F_0$-hermitian form
\[\phi:\ V\times V\rightarrow F\;,\]
which we assume to be split. This means that there exists a basis $e_1,\dots,e_n$ of $V$ such that
\[\phi(e_i,e_{n+1-j})=\delta_{i,j}\ \text{for all}\ i,j=1,\dots,n\;.\]
We have two associated $F_0$-bilinear forms:
\begin{gather*}
\langle x,y\rangle:=\frac12\tr_{F/F_0}(\pi^{-1}\;\phi(x,y))\;,\\
(x,y):=\frac12\tr_{F/F_0}(\phi(x,y))\;,
\end{gather*}
with \mbox{$\langle$ , $\rangle$} being alternating and \mbox{$($ , $)$} being symmetric. For any $\OO_F$-lattice $\Lambda$ in $V$ we denote by
\[\hat{\Lambda}:=\{v\in V;\ \phi(v,\Lambda)\subset\OO_F\}=
\{v\in V;\ \langle v,\Lambda\rangle\subset\OO_{F_0}\}\]
the dual lattice with respect to the alternating form and by
\[\hat{\Lambda}^{\text{s}}:=\{v\in V;\ (v,\Lambda)\subset\OO_{F_0}\}\]
the dual lattice with respect to the symmetric form. We have $\hat{\Lambda}^{\text{s}}=\pi^{-1}\;\hat{\Lambda}$.

For $i=0,\dots,n-1$, we define the standard lattices
\[\Lambda_i:=\spn_{\OO_F}\{\pi^{-1}e_1,\dots,\pi^{-1}e_i,e_{i+1},\dots,e_n\}\;.\]


\subsection{Selfdual Periodic Lattice Chain}\label{ssec:localmodel:latticechain}
Write $n=2m$ if $n$ is even and $n=2m+1$ if $n$ is odd. We consider nonempty subsets $I\subset\{0,\dots,m\}$ with the requirement that for $n=2m$ even, if $m-1$ is in $I$, then also $m$ is in $I$. We complete the $\Lambda_i$ with $i\in I$ to a selfdual periodic lattice chain by first including the duals $\Lambda_{n-i}:=\hat{\Lambda}^{\text{s}}_i$ for $i\in I\setminus\{0\}$ and then all the $\pi$-multiples: For $j\in\ZZ$ of the form $j=kn+i$ with $k\in\ZZ$ and $i\in I$ or $n-i\in I$, we set $\Lambda_j:=\pi^{-k}\;\Lambda_i$. Then the $\Lambda_j$ form a periodic lattice chain $\Lambda_I$, which satisfies $\hat{\Lambda}_j=\Lambda_{-j}$.

The index sets $I$ of the above form are in one-to-one correspondence with the parahoric subgroups of the unitary similitude group
\[\GU(V,\phi)=\{\matx{g}\in\GL_F(V);\ \phi(\matx{g}x,\matx{g}y)=c(\matx{g})\phi(x,y),\ c(\matx{g})\in {F_0}^{\times}\}\]
of the vector space $V$ and the form $\phi$, as is shown in sect.~1.b.3.\ of Pappas and Rapoport's paper \cite{PR:2007}. If $n=2m+1$ is odd, the correspondence is given by assigning the stabilizer subgroup
\[P_I:=\{\matx{g}\in\GU(V,\phi);\ \matx{g}\;\Lambda_i=\Lambda_i\ \text{for all}\ i\in I\}\subset \GU(V,\phi)\]
to the lattice chain $\Lambda_I$. If $n=2m$ is even, the situation is slightly more complicated. One has to consider a certain subgroup of $P_I$ (the kernel of the Kottwitz homomorphism), which gives a proper subgroup (of index two) exactly when $I$ does not contain $m$.


\subsection{Reflex Field}
Let ${F_0}^{\sep}$ be a fixed separable closure of $F_0$. We fix for each of the two embeddings $\varphi:F\to{F_0}^{\sep}$ an integer $r_{\varphi}$ with $0\le r_{\varphi}\le n$. The reflex field $E$ associated to these data is the finite field extension of $F_0$ contained in ${F_0}^{\sep}$ with
\[\Gal({F_0}^{\sep}/E)=\{\tau\in\Gal({F_0}^{\sep}/F_0);\;r_{\tau\varphi}=r_{\varphi}\ \text{for all}\ \varphi\}\;.\]


\subsection{Naive Local Model}
We fix nonnegative integers $r$ and $s$ with \mbox{$n=r+s$}. In the theory of Shimura varieties, these integers correspond to the signature of the algebraic group associated to the Shimura variety (after base change to the real numbers). By replacing $\phi$ by $-\phi$ if necessary, we may assume $s\le r$. We further assume $s>0$ (otherwise the corresponding Shimura variety is zero-dimensional). With $r$ and $s$ taken for $r_{\varphi}$ in the previous subsection, the reflex field $E$ equals $F$ if $r\ne s$ and $F_0$ if $r=s$.

For ease of notation, we denote the tensor product over $\OO_{F_0}$ just by $\ot$. We formulate a moduli problem $M_I^{\naive}$ on the category of $\OO_E$-schemes: A point of $M_I^{\naive}$ with values in an $\OO_E$-scheme $S$ is given by $\OO_F\ot\OO_S$-submodules
\[\F_j\subset\Lambda_j\ot\OO_S\]
for each $j\in\ZZ$ of the form $j=kn\pm i$ with $k\in\ZZ$ and $i\in I$. For each $j$, the following conditions have to be satisfied:
\begin{enumerate}
\item[(\hypertarget{hyp:localmodel:N1}{N1})] As an $\OO_S$-module, $\F_j$ is locally on $S$ a direct summand of rank $n$.
\item[(\hypertarget{hyp:localmodel:N2}{N2})] For each $j<j'$, there is a commutative diagram
\[\begin{array}{ccc}\Lambda_j\ot\OO_S&\rightarrow&\Lambda_{j'}\ot\OO_S\\
\cup&&\cup\\
\F_j&\rightarrow&\F_{j'}\end{array}\]
where the top horizontal map is induced by the lattice inclusion $\Lambda_j\subset\Lambda_{j'}$, and for each $j$, the isomorphism $\pi:\ \Lambda_j\rightarrow\Lambda_{j-n}$ induces an isomorphism of $\F_j$ with $\F_{j-n}$.
\item[(\hypertarget{hyp:localmodel:N3}{N3})] $\F_{-j}=\F_j^\perp$, with $\F_j^\perp$ denoting the orthogonal complement of $\F_j$ under the natural perfect pairing
\[{\langle\ ,\ \rangle}\ot\OO_S:\ (\Lambda_{-j}\ot\OO_S)\times(\Lambda_j\ot\OO_S)\rightarrow\OO_S\;.\]
\item[(\hypertarget{hyp:localmodel:N4}{N4})] Denote by $\Pi$ the respective action on $\Lambda_j\ot\OO_S$ given by multiplication with $\pi\ot1$. Since $\F_j$ is required to be an $\OO_F\ot\OO_S$-module, $\Pi$ restricts to an action on $\F_j$. The characteristic polynomial equals
\[\det(T\;\id-\Pi|\F_j)=(T-\pi)^s(T+\pi)^r\in\OO_S[T]\;.\]
\end{enumerate}
The moduli problem formulated in this way is representable by a projective scheme over $\Spec\OO_E$ since the above conditions define a closed subfunctor of a product of Grassmann functors.
$M^{\naive}_I$ is called the \emph{naive local model} associated to the group $GU(V,\phi)$, the signature type $(r,s)$, and the selfdual periodic lattice chain $\Lambda_I$.


\subsection{Wedge Local Model}
\label{ssec:localmodel:wedgelocalmodel}
As mentioned in the introduction, the naive local model is almost never flat over $\OO_E$. Pappas \cite{P:2000} defines a closed subscheme of $M_I^{\naive}$ by imposing an additional condition:
\begin{enumerate}
\item[(\hypertarget{hyp:localmodel:W}{W})] If $r\neq s$, we have for each $j$
\begin{gather*}
\wedge^{r+1}(\Pi-\sqrt{\pi_0}|\F_j)=0\;,\\
\wedge^{s+1}(\Pi+\sqrt{\pi_0}|\F_j)=0\;.
\end{gather*}
Here we have written $\sqrt{\pi_0}$ for the action on $\Lambda_j\ot\OO_S$ given by multiplication with $1\ot\pi$. Note that the assumption $r\neq s$ implies $\pi\in\OO_S$.
\end{enumerate}
We denote the corresponding closed subscheme by $M_I^{\wedge}$. It is called the \emph{wedge local model}.

\begin{lem}\label{lem:localmodel:genericfiber}
The wedge local model has the same generic fiber as the naive local model.
\end{lem}

\begin{proof}
We may assume $r\neq s$ since otherwise the wedge condition is trivial. To examine the generic fiber of the naive local model, we have to consider $A$-valued points, with $A$ an arbitrary $E$-algebra. These are given by subspaces $\F_j\subset\Lambda_j\ot A$ subject to conditions (\hyperlink{hyp:localmodel:N1}{N1})--(\hyperlink{hyp:localmodel:N4}{N4}). We fix an $\OO_F$-basis $f_1,\dots,f_n$ of $\Lambda_j$. This induces an $A$-basis $f_1,\pi f_1,\dots,f_n,\pi f_n$ of $\Lambda_j\ot A$ via the identification $\OO_F\iso\OO_{F_0}\cdot1+\OO_{F_0}\cdot\pi$. Then $\Pi$ is represented by the diagonal block matrix $\diag(\mat{B},\dots,\mat{B})$ of size $2n$, with the square matrix $\mat{B}$ of size two given by $\twomats{}{\pi_0}{1}{}$.

Since the characteristic polynomial of $\mat{B}$ is $T^2-\pi_0=(T-\pi)(T+\pi)$, the endomorphism $\Pi$ is diagonalizable over $A$. So is the restriction to the \mbox{$\Pi$-stable} subspace $\F_j$. By (\hyperlink{hyp:localmodel:N4}{N4}), the corresponding characteristic polynomial equals $(T-\pi)^s(T+\pi)^r$; hence, we can choose a basis such that $\Pi|\F_j$ is represented by the diagonal matrix $\diag(\pi,\dots,\pi,-\pi,\dots,-\pi)$, with $\pi$ occurring $s$ times and $-\pi$ occurring $r$ times. Now it is obvious that (\hyperlink{hyp:localmodel:W}{W}) is automatically satisfied in the situation considered. Therefore, the wedge condition does not alter the generic fiber.
\end{proof}
 

\subsection{Local Model}
The \emph{local model} $M_I^{\loc}$ is defined to be the flat closure of the generic fiber in the naive local model $M_I^{\naive}$. In particular, their generic fibers coincide. The following result will be used later on.

\begin{lem}\label{lem:localmodel:dimensiongenericfiber}
The generic fiber of the local model is irreducible of dimension $rs$.
\end{lem}

\begin{proof}
This is the statement of sect.~1.e.3.\ of \cite{PR:2007}.
\end{proof}

Because of Lem.~\ref{lem:localmodel:genericfiber}, the local model is also a closed subscheme of the wedge local model. Pappas and Rapoport \cite{PR:2007}*{Rem.\ 7.4)} give examples showing that in general the wedge condition is not sufficient to cut out the local model. In loc.cit., they propose one further condition (the so-called Spin condition) that should take care of this. Nevertheless, in some of the special cases we will consider below, the local model should already be given by the wedge local model; for a precise statement, see Rem.~\ref{rem:specialparahoric:flat}.


\section{Special Parahoric Level Structures}
We examine the local model $M_I^{\loc}$ for special choices of the index set $I$. If $n=2m+1$ is odd, we consider the cases $I=\{0\}$ and $I=\{m\}$; if $n=2m$ is even, we consider the case $I=\{m\}$. In sect.~1.b.3.\ of \cite{PR:2007}, it is shown that these are exactly the index sets for which the parahoric subgroups $P_I$ that preserve the lattice sets $\Lambda_i$ with $i\in I$ are \emph{special} in the sense of Bruhat--Tits theory \cite{T:1979}. The following theorem describes the special fiber of the corresponding local models.

\begin{thm}[main theorem]\label{thm:specialparahoric:main}
Let $I=\{0\}$ or $I=\{m\}$ if $n=2m+1$ is odd, and $I=\{m\}$ if $n=2m$ is even. Then the special fiber of the local model $M_I^{\loc}$ is irreducible and reduced; furthermore, the special fiber is normal, Frobenius split, and with only rational singularities.
\end{thm}

\begin{rem}\label{rem:specialparahoric:flat}
Pappas and Rapoport conjecture that under the assumptions of the main theorem, the wedge local model $M_I^{\wedge}$ is flat---provided that $s$ is even if $n$ is even. Confer Rem.~5.3 of their paper \cite{PR:2007}.
\end{rem}

\begin{rem}\label{rem:specialparahoric:Rapoport}
The cases $n=2m+1$ odd, $I=\{0\}$ and $n=2m$ even, $I=\{m\}$ have been treated by Pappas and Rapoport \cite{PR:2007}*{Thm.~5.1}. Calculations for the low-dimensional case $n=3$ odd, $I=\{1\}$ have been given in Prop.~6.2 of loc.cit. The arguments, however, cannot be generalized directly to the case of general $n=2m+1$ odd, $I=\{m\}$.
\end{rem}

\begin{proof}[Proof of the main theorem]
By Rem.~\ref{rem:specialparahoric:Rapoport}, we have to deal with the case of general $n=2m+1$ odd, $I=\{m\}$. Essentially, two results are required for the proof; these are obtained in the next two sections, where we first show that the special fiber of the local model contains an open subset that is reduced (Thm.~\ref{thm:reduced:reduced}) and then that the special fiber of the local model is irreducible (Thm.~\ref{thm:irreducibility:irreducible}).

The first result will be achieved by showing that the wedge local model contains an affine space of appropriate dimension as open subset (Prop.~\ref{prp:reduced:affine}), followed by some dimension arguments, which imply that this affine space is already lying in the local model. The second result will be deduced by considering an open neighborhood of a point which is contained in all irreducible components of the special fiber of the local model that can possibly exist (Prop.~\ref{prp:irreducibility:irreducibleneighborhood}). This is constructed by first examining the special fiber of the wedge local model and then intersecting with the local model.

Once we know that the special fiber of the local model is irreducible and generically reduced, the remaining assertions follow by standard arguments, as given by Pappas and Rapoport \cite{PR:2007}*{Proof of Thm.~5.1}. In particular, the main result of one of their previous papers is used to deduce the three properties ``normal, Frobenius split, and with only rational singularities'' \cite{PR:2006}*{Thm.~8.4}.
\end{proof}


\section{Open Reduced Subset of the Special Fiber}\label{sec:reduced}
Recall from the definition of the naive local model that we have fixed the signature type $(r,s)$ of the unitary group. The first result required in the proof of the main theorem is the next statement.

\begin{thm}[first step]\label{thm:reduced:reduced}
Let $n=2m+1$ be odd and $I=\{m\}$. Then the local model $M_I^{\loc}$ contains an affine space of dimension $rs$ as open subset. In particular, the special fiber of the local model contains an open subset that is reduced.
\end{thm}

We will first prove a corresponding statement for the wedge local model; from that, the theorem will be derived.

\begin{prp}\label{prp:reduced:affine}
Let $n=2m+1$ be odd and $I=\{m\}$. Then the wedge local model $M_I^{\wedge}$ contains an affine space of dimension $rs$ as open subset.
\end{prp}

\begin{proof}
Before starting the actual proof, which ranges over the remaining subsections, we introduce some matrices that will frequently occur from now on.

We write $\I_l$ for the unit matrix of size $l$
\[\I_l:=\threemat{1}{}{}{}{\ddots}{}{}{}{1}\]
and $\Hh_l$ for the unit antidiagonal matrix of size $l$
\[\Hh_l:=\threemat{}{}1{}{\iddots}{}{1}{}{}\;.\]
The matrix $\J_{k,l}$ is given by the antidiagonal matrix of size $k+l$
\[\J_{k,l}:=\twomat{}{\Hh_l}{-\Hh_k}{}\;.\]
The special case $k=l$ is abbreviated to $\J_{2k}:=\J_{k,k}$.


\subsection{Best Point}
\label{ssec:reduced:bestpoint}
Recall from sect.~\ref{ssec:localmodel:latticechain} the notion of the parahoric subgroup $P_I$: in the current situation of odd $n$, it is the stabilizer subgroup preserving the lattice chain $\Lambda_I$. This group acts on the special fibers of the models $M_I^{\naive}$, $M_I^{\wedge}$, and $M_I^{\loc}$. In sect.~3.c. of their paper, Pappas and Rapoport \cite{PR:2007} construct an embedding of the geometric special fiber of the naive local model into a partial affine flag variety (associated to the unitary similitude group).
This closed immersion is equivariant for the action of the parahoric, and thus its image is a union of Schubert varieties, which are enumerated by certain elements of the corresponding affine Weyl group.

In Prop.~3.1 of loc.cit., it is shown that the union of Schubert varieties over elements of the so-called $\mu$-admissible set is contained in the geometric special fiber of the local model. This union is denoted by $\AAA^I(\mu)$ in the notation of loc.cit.\ and is closed (since the $\mu$-admissible set is closed under the Bruhat order). In sect.~3.d of loc.cit., points of the local model are constructed that reduce to points lying in the Schubert varieties corresponding to the extreme elements of the $\mu$-admissible set. The open subset of the local model we are about to construct will contain one of these ``best points''.\footnote{A posteriori, we can see that in the situation under consideration, there is only a single extreme orbit, see Rem.~\ref{rem:irreducibility:oneextremeorbit}.}


\subsection{Conditions of the Wedge Local Model}\label{ssec:reduced:wedgelocalmodel}
We specialize the definition of the wedge local model to the case $n=2m+1$ odd, $I=\{m\}$. The essential part of the periodic lattice chain is given by
\[\ldots\rightarrow\Lambda_m\rightarrow\Lambda_{m+1}\rightarrow\ldots\;,\]
with $\Lambda_m$ and $\Lambda_{m+1}$ being the standard lattices
\begin{gather*}
\Lambda_m=\spn_{\OO_F}\{\pi^{-1}e_1,\dots,\pi^{-1}e_m,e_{m+1},\dots,e_n\}\;,\\
\Lambda_{m+1}=\spn_{\OO_F}\{\pi^{-1}e_1,\dots,\pi^{-1}e_{m+1},e_{m+2},\dots,e_n\}\;.
\end{gather*}
Denoting the above basis of $\Lambda_m$ by $f_1,\dots,f_n$ and that of $\Lambda_{m+1}$ by $g_1,\dots,g_n$, we have corresponding $\OO_{F_0}$-bases $f_1,\dots,f_n,\pi f_1,\dots,\pi f_n$ and $g_1,\dots,g_n$,$\pi g_1,\dots,\pi g_n$, respectively.

We have to examine $A$-valued points of $M_I^{\wedge}$, with $A$ an arbitrary $\OO_E$-algebra. This means considering $\OO_F\ot A$-submodules
\begin{gather*}
\F\subset\Lambda_m\ot A\;,\\
\G\subset\Lambda_{m+1}\ot A
\end{gather*}
subject to the conditions of the wedge local model. These translate into:
\begin{enumerate}
\item[(\hypertarget{hyp:reduced:N1}{N1})] As $A$-modules, $\F$ and $\G$ are locally direct summands of rank $n$. Identifying $\Lambda_m\ot A$ and $\Lambda_{m+1}\ot A$ with $A^{2n}$ via the above $\OO_{F_0}$-bases, we can consider $\F$ and $\G$ as $A$-valued points of the Grassmannian $\Grass_{n,2n}$.
\item[(\hypertarget{hyp:reduced:N2}{N2})] The maps induced by the inclusions $\Lambda_m\subset\Lambda_{m+1}$ and $\Lambda_{m+1}\subset\pi^{-1}\Lambda_m$ restrict to maps
\[\F\to\G\to\pi^{-1}\F\;.\]
Here $\pi^{-1}\F$ is the image of $\F$ under the map induced by the isomorphism $\pi^{-1}:\ \Lambda_m\to\pi^{-1}\Lambda_m$. 
\item[(\hypertarget{hyp:reduced:N3}{N3})] $\G=\F^\perp$, with $\F^{\perp}$ denoting the orthogonal complement of $\F$ under the natural perfect pairing
\beqn\label{eq:reduced:perfpair}(\mbox{ , })\ot A:\;(\Lambda_{m+1}\ot A)\times(\Lambda_m\ot A)\rightarrow A\;.\eeqn
With respect to the chosen bases, the form is represented by the $2n\times2n$-matrix
\[\mat{M}:=\twomat{}{-\J_{m,m+1}}{\J_{m,m+1}}{}\;.\]
\item[(\hypertarget{hyp:reduced:N4}{N4})] The characteristic polynomial of $\Pi|\F$ is given by
\[\det(T\;\id-\Pi|\F)=(T-\pi)^s(T+\pi)^r\in A[T]\;,\]
and the analogous statement holds true for $\G$.
\item[(\hypertarget{hyp:reduced:W}{W})] We have
\begin{gather*}
\wedge^{r+1}(\Pi-\sqrt{\pi_0}|\F)=0\;,\\
\wedge^{s+1}(\Pi+\sqrt{\pi_0}|\F)=0\;,
\end{gather*}
and the analogous statement holds true for $\G$.
\end{enumerate}
Viewing $\F$ and $\G$ as $A$-modules, the fact that they are required to be modules over $\OO_F\ot A$ translates into an additional condition:
\begin{enumerate}
\item[(\hypertarget{hyp:reduced:Pi}{Pi})] $\F$ and $\G$ are $\Pi$-stable.
\end{enumerate}
These conditions will be evaluated in the following subsections.


\subsection{Orthogonal Complement}
Condition (\hyperlink{hyp:reduced:N3}{N3}) implies that the subspace $\G$ is determined by $\F$ as its orthogonal complement. We denote by $W$ the corresponding subfunctor of $\Grass_{n,2n}\times\Grass_{n,2n}$ that satisfies (\hyperlink{hyp:reduced:N3}{N3}). Then the projection onto the first factor, $\pr_{\F}:\ \Grass_{n,2n}\times\Grass_{n,2n}\to \Grass_{n,2n}$, restricts to an isomorphism of functors:
\[{\pr_{\F}}|W:\ W\stackrel{\sim}{\rightarrow}\Grass_{n,2n}\;.\]
This is because the assignment $\F\mapsto(\F,\F^{\perp})$ on $A$-valued points induces an inverse morphism, as can be seen from the explicit determination of the orthogonal complement in Lem.~\ref{lem:reduced:orthogonalcomplement} below. Since our objective is to construct an open subset of the wedge local model, we may restrict ourselves to considering subfunctors of $W$ that are induced via the isomorphism $\pr_{\F}$ by open subfunctors of $\Grass_{n,2n}$.

Recall that the Grassmann functor is covered by the open subfunctors
\begin{align*}
\Grass_{n,2n}^J(A):=\{&\U\in\Grass_{n,2n}(A);\ \\ &\OO_A^J\hookrightarrow\OO_A^{2n}\twoheadrightarrow\OO_A^{2n}/\U\ \text{is an isomorphism}\}\;,
\end{align*}
where $J\subset\{1,\dots,2n\}$ is a subset consisting of $n$ elements, and the arrows denote the obvious homomorphisms. These functors are represented by affine space of dimension $n^2$.

We consider the complement $J$ of the index set that corresponds to the basis elements $f_1,\dots,f_s$, $\pi f_1,\dots,\pi f_r$ (for a motivation of this choice, see Rem.~\ref{rem:reduced:bestpoint}). The elements of $\Grass_{n,2n}^J(A)$ can be described as the column span of $2n\times n$-matrices~$\F$ having entries in $A$ and being of the following form:
\beqn\label{eq:reduced:F}\F=\fourtwomat{\I_s}{}{\mat{a}}{\mat{b}}{}{\I_r}{\mat{c}}{\mat{d}}\;.\eeqn
Here the submatrix $\mat{a}$ has $r$ rows and $s$ columns, $\mat{d}$ has $s$ rows and $r$ columns, and as usual, $\I_s$ and $\I_r$ are the unit matrices of sizes $s$ and $r$, respectively. An impression of the ratio of the respective blocks can be received from Fig.~\ref{fig:reduced:matrix}.
\begin{figure}
\includegraphics{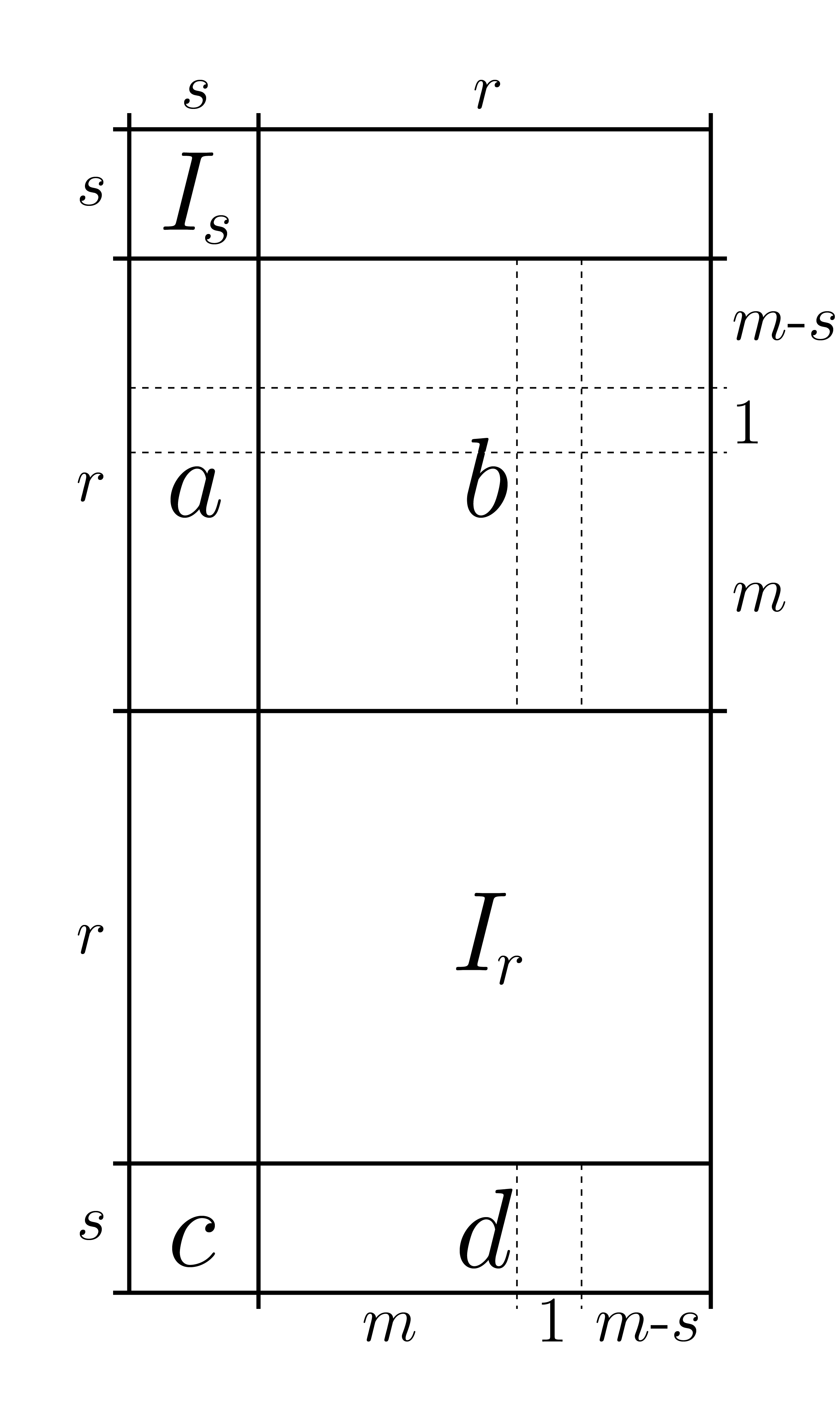}
\caption{Typical form of the matrix $\F$ (for $n=2m+1$ odd). The partitioning shown corresponds to $n=9$ and $s=2$ (then $m=4$ and $r=5$). The solid lines separate the main blocks; the dotted lines indicate a finer subdivision helpful for the upcoming calculations. The labels outside denote the sizes of the blocks.}
\label{fig:reduced:matrix}
\end{figure}
We denote the subspace $\F$ and the matrix representing it (as a column span) by the same symbol. This should not lead to any confusion, since the intended meaning will be clear from the context.

To describe the orthogonal complement of $\F$ in a clear way, we introduce further notations. For the moment, let $\mat{B}$ be an arbitrary matrix with $k$ rows and $l$ columns. We define the involution $\iota$ as follows:
\[\iota(\mat{B}):=\Hh_l\;\mat{B}^{\T}\;\Hh_k\;.\]
This is the matrix obtained from $\mat{B}$ by reflection at the first angle bisector going through the lower left matrix entry (which is precisely the antidiagonal in the case of a square matrix).
Assuming $i\le k$, we denote by $\mat{B}^{[i]}$ the matrix consisting of the first $i$ rows of $B$ and by $\mat{B}_{[i]}$ the matrix consisting of the last $i$ rows. The $i$th row is denoted by $\mat{B}^{(i)}$. Likewise, assuming $j\le l$, we write $^{[j]}\mat{B}$, $_{[j]}\mat{B}$, and $^{(j)}\mat{B}$ for the first $j$ columns of $B$, the last $j$ columns, and the $j$th column, respectively. We refer to the single matrix entry in the $i$th row and $j$th column as $\mat{B}_{i,j}$.

\begin{lem}\label{lem:reduced:orthogonalcomplement}
With respect to the perfect pairing \eqref{eq:reduced:perfpair}, the orthogonal complement of $\F$ is given by the column span of the matrix
\[\G=\fourtwomat{\I_s}{}{\mat{\tilde{a}}}{\mat{\tilde{b}}}{}{\I_r}{\mat{\tilde{c}}}{\mat{\tilde{d}}}\;,\]
with
\begin{gather*}
\mat{\tilde{a}}=\twovec{-\iota(_{[r-m]}\mat{d})}{\iota(^{[m]}\mat{d})}\;,\\
\mat{\tilde{b}}=\twomat{\iota(_{[r-m]}\mat{b}_{[m+1]})}{-\iota(_{[r-m]}\mat{b}^{[m-s]})}{-\iota(^{[m]}\mat{b}_{[m+1]})}{\iota(^{[m]}\mat{b}^{[m-s]})}\;,\\
\mat{\tilde{c}}=-\iota(\mat{c})\;,\\
\mat{\tilde{d}}=\twovex{\iota(\mat{a}_{[m+1]})}{-\iota(\mat{a}^{[m-s]})}\;.
\end{gather*}
\end{lem}

\begin{proof}
$\G$ is a subspace of rank $n$, and one calculates $\G^{\T}\;\mat{M}\;\F=\I_{2n}$ (recall that $M$ is the matrix representing the perfect pairing).
\end{proof}

\begin{rem}\label{rem:reduced:bestpoint}
It can be easily checked that $(\F_1,\G_1)\in\Grass_{n,2n}^J(k)\times\Grass_{n,2n}^J(k)$,
given by the $k$-subspaces
\begin{gather*}
\F_1:=\spn_k\{f_1,\dots,f_s,\pi f_1,\dots,\pi f_r\}\;,\\
\G_1:=\spn_k\{g_1,\dots,g_s,\pi g_1,\dots,\pi g_r\}\;,
\end{gather*}
satisfies the conditions of the wedge local model and, thus, represents a point of the special fiber of the wedge local model (in the above notation, $\F_1$ corresponds to $a=b=c=d=0$ and $\G_1$ corresponds to $\tilde{a}=\tilde{b}=\tilde{c}=\tilde{d}=0$). More precisely, this is one of the special points mentioned in sect.~\ref{ssec:reduced:bestpoint}: this follows from sect.~3.d of \cite{PR:2007} by considering (in the notation of loc.cit.) the subset $S=[n+1-s,n]$. It follows that $(\F_1,\G_1)$ is lying in the special fiber of the local model.
\end{rem}


\subsection{Pi-Stability}\label{ssec:reduced:pi-stability}
We continue to evaluate the conditions of the wedge local model. We are given pairs of subspaces $(\F,\F^{\perp})$. Condition (\hyperlink{hyp:reduced:Pi}{Pi}), concerning the stability of $\F$ under the action of $\Pi$, translates into the equation
\beqn\label{eq:reduced:PiStable}\Pi\;\F=\F\;\mat{R}\;.\eeqn
Here $\mat{R}$ is a square matrix of size $n$, which we subdivide into four blocks as follows:
\[\mat{R}=\twomat{\mat{S}}{\mat{T}}{\mat{U}}{\mat{V}}\;,\]
with $\mat{S}$ a square matrix of size $s$ and $\mat{V}$ a square matrix of size $r$. With respect to the chosen basis, the operator $\Pi$ is given by the matrix
\[\mat{\Pi}=\twomat{}{\pi_0\;\I_n}{\I_n}{}\;.\]
Then \eqref{eq:reduced:PiStable} comes to
\beqn\label{eq:reduced:idPiStable}\fourtwomat{\mat{0}}{\pi_0\;\I_r}{\pi_0\;\mat{c}}{\pi_0\;\mat{d}}{\I_s}{\mat{0}}{\mat{a}}{\mat{b}} = \fourtwomat{\mat{S}}{\mat{T}}{\mat{a}\;\mat{S}+\mat{b}\;\mat{U}}{\mat{a}\;\mat{T}+\mat{b}\;\mat{V}}{\mat{U}}{\mat{V}}{\mat{c}\;\mat{S}+\mat{d}\;\mat{U}}{\mat{c}\;\mat{T}+\mat{d}\;\mat{V}}\;.\eeqn
Comparison of the matrices yields several identities involving the $a$-, $b$-, $c$-, and $d$-variables. This has to be done carefully since the blocks of the matrices that seem to correspond are of different sizes.

To begin with, we obtain from \eqref{eq:reduced:idPiStable} the following identities concerning the blocks of the matrix $\mat{R}$:
\beqn\label{eq:reduced:idSTUV}\mat{S}=\mat{0}\;,\quad\mat{T}=\pi_0\;{\I_r}^{[s]}\;,\quad\mat{U}=\twovec{\I_s}{\mat{a}^{[r-s]}}\;,\quad\mat{V}=\twovec{\mat{0}}{\mat{b}^{[r-s]}}\;.\eeqn
Thus, the matrix $\mat{R}$ takes the form
\beqn\label{eq:reduced:R}\mat{R}=\threemat{}{\pi_0\;\I_s}{}{\I_s}{}{}{\mat{a}^{[r-s]}}{^{[s]}\mat{b}^{[r-s]}}{_{[r-s]}\mat{b}^{[r-s]}}\;.\eeqn


\subsection{Wedge Condition}\label{ssec:reduced:wedgecondition}
Before examining the remaining blocks of \eqref{eq:reduced:idPiStable}, we take a look at the wedge condition (\hyperlink{hyp:reduced:W}{W}).

Since $\Pi|\F$ is given by the matrix $\mat{R}$, all minors of size $r+1$ of
\beqn\label{eq:reduced:R-}\mat{R}-\pi\;\I_n=\threemat{-\pi\;\I_s}{\pi_0\;\I_s}{}{\I_s}{-\pi\;\I_s}{}{\mat{a}^{[r-s]}}{^{[s]}\mat{b}^{[r-s]}}{_{[r-s]}\mat{b}^{[r-s]}-\pi\;\I_{r-s}}\eeqn
have to be zero. Note that the first $s$ rows are multiples of the following $s$ rows. Since any minor of size $r+1$ includes at least one pair of such corresponding rows, all these minors are zero.

All minors of size $s+1$ of
\beqn\label{eq:reduced:R+}\mat{R}+\pi\;\I_n=\threemat{\pi\;\I_s}{\pi_0\;\I_s}{}{\I_s}{\pi\;\I_s}{}{\mat{a}^{[r-s]}}{^{[s]}\mat{b}^{[r-s]}}{_{[r-s]}\mat{b}^{[r-s]}+\pi\;\I_{r-s}}\eeqn
have to be zero as well. First, we consider the minors of size $s+1$
obtained by keeping only the rows with row number in $\{s+1,\dots,2s,2s+i\}$ and the columns with column number in $\{1,\dots,s,s+j\}$. Here $i$ and $j$ denote integers with $1\le i\le r-s$ and $1\le j\le s$. We use Laplace expansion along the last column and calculate
\[\det\twomat{\I_s}{\pi\;{^{(j)}\I_s}}{\mat{a}^{(i)}}{\mat{b}_{i,j}}=(-1)^{2(s+1)}\mat{b}_{i,j}+(-1)^{s+1+j}(-1)^{s-j}\pi\;\mat{a}_{i,j}\;.\]
These minors being zero, we get
\beqn\label{eq:reduced:s_b_r-s}^{[s]}\mat{b}^{[r-s]}=\pi\;\mat{a}^{[r-s]}\;.\eeqn
Next, we consider the minors obtained by keeping the rows $\{s+1,\dots,2s,2s+i\}$ and the columns $\{1,\dots,s,2s+j\}$, with $1\le i,j\le r-s$:
\[\det\twomat{\I_s}{0}{\mat{a}^{(i)}}{(_{[r-s]}\mat{b}^{[r-s]}+\pi\;\I_{r-s})_{i,j}}=(_{[r-s]}\mat{b}^{[r-s]}+\pi\;\I_{r-s})_{i,j}\;.\]
These minors being zero, we obtain
\beqn\label{eq:reduced:r-s_b_r-s} _{[r-s]}\mat{b}^{[r-s]}=-\pi\;\I_{r-s}\;.\eeqn
Finally, all remaining minors of size $s+1$ are automatically zero now.


\subsection{Characteristic Polynomial}
The characteristic polynomial of $\Pi|\F$ is given by $\det(T\;\I_n-\mat{R})$, with $R$ as in $\eqref{eq:reduced:R}$. Making use of \eqref{eq:reduced:r-s_b_r-s}, we calculate
\[\det\threemat{T\;\I_s}{-\pi_0\;\I_s}{}{-\I_s}{T\;\I_s}{}{-\mat{a}^{[r-s]}}{-{^{[s]}\mat{b}^{[r-s]}}}{(T+\pi)\;\I_{r-s}}=(T-\pi)^s(T+\pi)^r\in A[T]\;,\]
which is in accordance with (\hyperlink{hyp:reduced:N4}{N4}).


\subsection{Pi-Stability (continued)}\label{ssec:reduced:pi-stability2}
We show that the $b$-variables are determined by the $a$-, $c$-, and $d$-variables. For this purpose, we consider the matrix equation $\mat{c}\;\mat{T}+\mat{d}\;\mat{V}=\mat{b}_{[s]}$,
obtained from the lower right blocks of the matrices in \eqref{eq:reduced:idPiStable}.
Using \eqref{eq:reduced:idSTUV}, \eqref{eq:reduced:s_b_r-s}, and \eqref{eq:reduced:r-s_b_r-s}, the first $s$ columns give
\beqn\label{eq:reduced:s_b_s} ^{[s]}\mat{b}_{[s]}=\pi_0\;\mat{c}+{_{[r-s]}\mat{d}}\;^{[s]}\mat{b}^{[r-s]}=\pi_0\;\mat{c}+\pi\;{_{[r-s]}\mat{d}}\;\mat{a}^{[r-s]}\;,\eeqn
and the last $r-s$ columns give
\beqn\label{eq:reduced:r-s_b_s} _{[r-s]}\mat{b}_{[s]}={_{[r-s]}\mat{d}}\;_{[r-s]}\mat{b}^{[r-s]}=-\pi\;{_{[r-s]}\mat{d}}\;.\eeqn
Combining \eqref{eq:reduced:s_b_r-s}--\eqref{eq:reduced:r-s_b_s} yields the following description of the submatrix $\mat{b}$:
\beqn\label{eq:reduced:b}\mat{b}=\twomat{\pi\;\mat{a}^{[r-s]}}{-\pi\;\I_{r-s}}{\pi_0\;\mat{c}+\pi\;{_{[r-s]}\mat{d}}\;\mat{a}^{[r-s]}}{-\pi\;{_{[r-s]}\mat{d}}}\;.\eeqn
Hence, the $\mat{b}$-variables are determined by the other variables.

With \eqref{eq:reduced:idSTUV}, the lower left blocks of the matrices in \eqref{eq:reduced:idPiStable} give the identity 
\beqn\label{eq:reduced:a_s} ^{[s]}\mat{d}=\mat{a}_{[s]}-{_{[r-s]}\mat{d}}\;\mat{a}^{[r-s]}\;,\eeqn
to which we return later. 

The remaining blocks of the matrices in \eqref{eq:reduced:idPiStable} give nothing new.


\subsection{Lattice Inclusion Map}\label{ssec:reduced:latticeinclusion}
We can deduce further constraints on the $a$-, $c$-, and $d$-variables from  (\hyperlink{hyp:reduced:N2}{N2}), which demands that the maps induced by the lattice inclusions restrict to the considered subspaces.

With respect to the chosen bases, the map corresponding to $\Lambda_m\subset\Lambda_{m+1}$ is given by the $2n\times2n$-matrix
\[\mat{A}:=\begin{pmatrix}\I_m\\&0&&&\pi_0\\&&\I_m\\&&&\I_m\\&1&&&0\\&&&&&\I_m\end{pmatrix}\;.\]
Since this map is required to restrict to $\F\rightarrow\F^{\perp}$, we have to examine the conditions under which $\mat{A}\;\F$ is perpendicular to $\F$. With $\mat{M}$ as in (\hyperlink{hyp:reduced:N3}{N3}), $\mat{C}:=\F^{\T}\;\mat{A}^{\T}\;\mat{M}\;\F$ has to be the zero matrix of size $n$. We multiply the matrices on the right hand side. Note the form of the matrix
\[\mat{A}^{\T}\;\mat{M}=\begin{pmatrix}&&&&&-\Hh_m\\&1&&&0\\&&&\Hh_m\\&&\Hh_m\\&0&&&-\pi_0\\-\Hh_m\end{pmatrix}\;,\]
which suggests that blockwise multiplying becomes easier when subdividing $\F$ into four groups of columns, with the groups consisting of $s$, $m$, $1$, and $m-s$ columns. This partitioning is shown in Fig.~\ref{fig:reduced:matrix}. The symmetry of $\mat{A}^{\T}\;\mat{M}$ implies the symmetry of $\mat{C}$, and we obtain ten conditions from the blocks of $\mat{C}$:
\begin{enumerate}
\item[(\hypertarget{hyp:reduced:C1}{C1})] $0=-\Hh_s\;\mat{c}+{\mat{a}^{(m-s+1)}}^{\T}\;\mat{a}^{(m-s+1)}-\mat{c}^{\T}\;\Hh_s$\,,
\item[(\hypertarget{hyp:reduced:C2}{C2})] $0=-\Hh_s\;^{[m]}\mat{d}+{\mat{a}^{(m-s+1)}}^{\T}\;^{[m]}\mat{b}^{(m-s+1)}+{\mat{a}_{[m]}}^{\T}\;\Hh_m$\,,
\item[(\hypertarget{hyp:reduced:C3}{C3})] $0=-\Hh_s\;^{(m+1)}\mat{d}+{\mat{a}^{(m-s+1)}}^{\T}\;\mat{b}_{m-s+1,m+1}$\,,
\item[(\hypertarget{hyp:reduced:C4}{C4})] $0=-\Hh_s\;_{[m-s]}\mat{d}-{\mat{a}^{[m-s]}}^{\T}\;\Hh_{m-s}+{\mat{a}^{(m-s+1)}}^{\T}\;_{[m-s]}\mat{b}^{(m-s+1)}$\,,
\item[(\hypertarget{hyp:reduced:C5}{C5})] $0={^{[m]}\mat{b}^{(m-s+1)}}^{\T}\;^{[m]}\mat{b}^{(m-s+1)}+{^{[m]}\mat{b}_{[m]}}^{\T}\;\Hh_m+\Hh_m\;^{[m]}\mat{b}_{[m]}$\,,
\item[(\hypertarget{hyp:reduced:C6}{C6})] $0={^{[m]}\mat{b}^{(m-s+1)}}^{\T}\;\mat{b}_{m-s+1,m+1}+\Hh_m\;^{(m+1)}\mat{b}_{[m]}$\,,
\item[(\hypertarget{hyp:reduced:C7}{C7})] $0=-{^{[m]}\mat{b}^{[m-s]}}^{\T}\;\Hh_{m-s}+{^{[m]}\mat{b}^{(m-s+1)}}^{\T}\;_{[m-s]}\mat{b}^{(m-s+1)}+\Hh_m\;_{[m-s]}\mat{b}_{[m]}$\,,
\item[(\hypertarget{hyp:reduced:C8}{C8})] $0={\mat{b}_{m-s+1,m+1}}^2-\pi_0$\,,
\item[(\hypertarget{hyp:reduced:C9}{C9})] $0=-{^{(m+1)}\mat{b}^{[m-s]}}^{\T}\;\Hh_{m-s}+\mat{b}_{m-s+1,m+1}\;_{[m-s]}\mat{b}^{(m-s+1)}$\,,
\item[(\hypertarget{hyp:reduced:C10}{C10})] $0=-{_{[m-s]}\mat{b}^{[m-s]}}^{\T}\;\Hh_{m-s}+{_{[m-s]}\mat{b}^{(m-s+1)}}^{\T}\;_{[m-s]}\mat{b}^{(m-s+1)}-{\Hh_{m-s}\;_{[m-s]}\mat{b}^{[m-s]}}$\,.
\end{enumerate}

These conditions will now be evaluated, beginning with (\hyperlink{hyp:reduced:C1}{C1}). We collect the \mbox{$c$-variables} on the left hand side and left-multiply with $\Hh_s$ to obtain
\begin{enumerate}
\item[(\hypertarget{hyp:reduced:C1'}{C$1'$})] $\mat{c}+\iota(\mat{c})=\Hh_s\;{\mat{a}^{(m-s+1)}}^{\T}\;\mat{a}^{(m-s+1)}$\,.
\end{enumerate}
Both sides of the last equation are symmetric with respect to reflection at the antidiagonal (that is, invariant under the involution $\iota$). Therefore, it suffices to look at entries on or above the antidiagonal; these are the entries indexed by $(i,j)$ with $1\leq i\leq s$, $1\leq j\leq s+1-i$. Note that only $a$-variables occur on the right hand side of (\hyperlink{hyp:reduced:C1'}{C$1'$}), which we temporarily denote by $\mat{B}$. We get equations of the form
\[\mat{c}_{i,j}+\mat{c}_{s+1-j,s+1-i}=\mat{B}_{i,j}\;.\]
The entries on the antidiagonal give $\mat{c}_{i,s+1-i}=\mat{B}_{i,s+1-i}/2$ (by assumption, the characteristic is $\neq2$); those above the antidiagonal give $\mat{c}_{s+1-j,s+1-i}=\mat{B}_{i,j}-\mat{c}_{i,j}$. Hence, we may keep the elements of the set
\[\{\mat{c}_{i,j};\ 1\leq i<s,\ 1\leq j<s+1-i\}\]
as free variables, determining (together with the $a$-variables) all remaining $\mat{c}_{i,j}$ with $1\leq i\leq s$, $s+1-i\leq j\leq s$. The free $c$-variables are ${s(s-1)}/2$ in number.

Analogously, we rearrange (\hyperlink{hyp:reduced:C2}{C2})--(\hyperlink{hyp:reduced:C4}{C4}) to get
\begin{enumerate}
\item[(\hypertarget{hyp:reduced:C2'}{C$2'$})] $^{[m]}\mat{d}=\iota(\mat{a}_{[m]})+\Hh_s\;{\mat{a}^{(m-s+1)}}^{\T}\;\twovex{\pi\;\mat{a}^{(m-s+1)}}{\mat{0}}$\,,
\item[(\hypertarget{hyp:reduced:C3'}{C$3'$})] $^{(m+1)}\mat{d}=-\pi\;\iota(\mat{a}^{(m-s+1)})$\,,
\item[(\hypertarget{hyp:reduced:C4'}{C$4'$})] $_{[m-s]}\mat{d}=-\iota(\mat{a}^{[m-s]})$\,.
\end{enumerate}
Consequently, the $d$-variables are determined by the $a$-variables.

We split $_{[r-s]}\mat{d}\;\mat{a}^{[r-s]}$ into three terms,
\beqn\label{eq:reduced:r-sd}
\begin{aligned}
_{[r-s]}\mat{d}\;\mat{a}^{[r-s]}=&{_{[m-s]}{^{[m]}\mat{d}}}\;\mat{a}^{[m-s]}+{^{(m+1)}\mat{d}}\;\mat{a}^{(m-s+1)}\\
&+{_{[m-s]}\mat{d}}\;{\mat{a}_{[m]}}^{[m-s]}\;,
\end{aligned}
\eeqn
with which we substitute the corresponding term in equation \eqref{eq:reduced:a_s} (this equation has not been considered yet). We then use (\hyperlink{hyp:reduced:C2'}{C$2'$})--(\hyperlink{hyp:reduced:C4'}{C$4'$}) to replace the $d$-variables and obtain after rearranging
\beqn\label{eq:reduced:maineq} \mat{a}_{[s]}-\iota(\mat{a}_{[s]})=\iota({\mat{a}_{[m]}}^{[m-s]})\;\mat{a}^{[m-s]}-\iota(\mat{a}^{[m-s]})\;{\mat{a}_{[m]}}^{[m-s]}\;.\eeqn
All elements $\mat{a}_{i,j}$ on the right hand side have index $(i,j)$ in the set
\[\mathcal{I}:=\{(i,j);\ 1\leq i\leq r-s,\ 1\leq j\leq s\}\;,\]
whereas all elements on the left hand side have index $(i,j)$ in the complement
\[\mathcal{Q}:=\{(i,j);\ r-s+1\leq i\leq r,\ 1\leq j\leq s\}\;.\]
Both sides of \eqref{eq:reduced:maineq} are antisymmetric with respect to reflection at the antidiagonal. We argue as above (in the case of the $c$-variables) and keep the elements of the set
\[\{\mat{a}_{i,j};\ (i,j)\in\mathcal{I}\}\cup\{\mat{a}_{i,j};\ r-s+1\leq i\leq r,\ 1\leq j\leq r+1-i\}\]
as free variables. They determine the remaining $\mat{a}_{i,j}$ with $r-s+1<i\leq r$, $r+1-i<j\leq s$. Hence, there are $s(r-s)+{s(s+1)}/2$ free $a$-variables.

Since the $c$-variables are independent of the $a$-variables, we conclude that the pairs $(\F,\G)$ satisfying the conditions so far describe an affine space of dimension
\[\frac{s(s-1)}2+s(r-s)+\frac{s(s+1)}2=rs\;,\]
which is in accordance with the assertion of the proposition.


\subsection{Remaining Conditions}
If we could show that all remaining conditions of the wedge local model are fulfilled by now, the proposition would follow.

\begin{rem}\label{rem:reduced:remainingconditions}
In the proof of Thm.~\ref{thm:reduced:reduced} we give a dimension argument, which implies that the remaining conditions are automatically satisfied; see Rem.~\ref{rem:reduced:conditionssufficed}.
\end{rem}

Despite the remark, we indicate how this can be verified by direct calculations.

We start with equations (\hyperlink{hyp:reduced:C5}{C5})--(\hyperlink{hyp:reduced:C10}{C10}), which involve only $b$-variables. They are automatically satisfied, as can be shown by using the description \eqref{eq:reduced:b} of $\mat{b}$ together with (\hyperlink{hyp:reduced:C1'}{C$1'$})--(\hyperlink{hyp:reduced:C4'}{C$4'$}). To give an example, we check that (\hyperlink{hyp:reduced:C5}{C5}) holds true. Rearranging and multiplying with $\Hh_m$ from the left yields
\[^{[m]}\mat{b}_{[m]}+\iota(^{[m]}\mat{b}_{[m]})=-\Hh_m\;{^{[m]}\mat{b}^{(m-s+1)}}^{\T}\;^{[m]}\mat{b}^{(m-s+1)}\;.\]
With \eqref{eq:reduced:b}, we see that the left hand side is the $\pi$-fold of the matrix
\[\twomat{{\mat{a}^{[r-s]}}_{[m-s]}-\iota(^{[m-s]}{_{[r-s]}\mat{d}})}{\mat{0}}{\pi\;(\mat{c}+\iota(\mat{c}))+{_{[r-s]}\mat{d}}\;\mat{a}^{[r-s]}+\iota(_{[r-s]}\mat{d}\;\mat{a}^{[r-s]})}{\iota({\mat{a}^{[r-s]}}_{[m-s]})-{^{[m-s]}{_{[r-s]}\mat{d}}}}\]
and the right hand side the $\pi$-fold of the matrix
\[\twomat{\mat{0}}{\mat{0}}{-\pi\;\Hh_s\;{\mat{a}^{(m-s+1)}}^{\T}\;\mat{a}^{(m-s+1)}}{\mat{0}}\;.\]
For the moment, we refer to these matrices by $B_{\lhs}$ and $B_{\rhs}$, respectively. Note that ${\mat{a}^{[r-s]}}_{[m-s]}$ denotes the same matrix block as ${\mat{a}_{[m]}}^{[m-s]}$; in the same manner,  $^{[m-s]}{_{[r-s]}\mat{d}}={_{[m-s]}{^{[m]}\mat{d}}}$. Then (\hyperlink{hyp:reduced:C2'}{C$2'$}) implies that the upper left block and the lower right block of $B_{\lhs}$ are zero. The lower left blocks of $B_{\lhs}$ and $B_{\rhs}$ also coincide. This follows from (\hyperlink{hyp:reduced:C1'}{C$1'$}) and the calculation
\[_{[r-s]}\mat{d}\;\mat{a}^{[r-s]}+\iota(_{[r-s]}\mat{d}\;\mat{a}^{[r-s]})=-2\pi\;\Hh_s\;{\mat{a}^{(m-s+1)}}^{\T}\;\mat{a}^{(m-s+1)}\;.\]
For the last equation, we have used \eqref{eq:reduced:r-sd} and (\hyperlink{hyp:reduced:C2'}{C$2'$})--(\hyperlink{hyp:reduced:C4'}{C$4'$}).

Next, we consider the $\Pi$-stability of $\F^{\perp}$. The adjoint of $\Pi$ (with respect to the perfect pairing \eqref{eq:reduced:perfpair}) is given by $\Pi^{\ast}=-\Pi$ since $\mat{\Pi}^{\T}\;\mat{M}=-\mat{M}\;\mat{\Pi}$. With $x\in\F^{\perp}$, we then have $\Pi\;x\in\F^{\perp}$ as well. Indeed, we calculate $(\Pi\;x,y)=(x,-\Pi\;y)=0$ for all $y\in\F$ because $\Pi\;y\in\F$.

Analogously to \eqref{eq:reduced:R}, the matrix
\beqn\label{eq:reduced:R'}\mat{R'}:=\threemat{}{\pi_0\;\I_s}{}{\I_s}{}{}{\mat{\tilde{a}}^{[r-s]}}{^{[s]}\mat{\tilde{b}}^{[r-s]}}{_{[r-s]}\mat{\tilde{b}}^{[r-s]}}\eeqn
describes the action of $\Pi$ on $\F^{\perp}$. Here the \~{}-signs denote the corresponding entries of the orthogonal complement of $\F$, given in Lem.~\ref{lem:reduced:orthogonalcomplement}. Together with the description \eqref{eq:reduced:b} of $\mat{b}$, it follows that we have $_{[r-s]}\mat{\tilde{b}}^{[r-s]}=-\pi\;\I_{r-s}$.
Hence, the characteristic polynomial of $\Pi|\F^{\perp}$ equals
\[\det\threemat{T\;\I_s}{-\pi_0\;\I_s}{}{-\I_s}{T\;\I_s}{}{-\mat{\tilde{a}}^{[r-s]}}{-{^{[s]}\mat{\tilde{b}}^{[r-s]}}}{(T+\pi)\;\I_{r-s}}=(T-\pi)^s(T+\pi)^r\in A[T]\;,\]
and therefore (\hyperlink{hyp:reduced:N4}{N4}) holds true.

Condition (\hyperlink{hyp:reduced:W}{W}) for $\F^{\perp}$ is also satisfied, as can be seen in complete analogy to the calculations in sect.~\ref{ssec:reduced:wedgecondition}.

Less obvious is the last requirement to be checked, namely, that the map induced by the lattice inclusion $\Lambda_{m+1}\subset\pi^{-1}\Lambda_m$ restricts to a map $\F^{\perp}\to\pi^{-1}\F$. The former map is given by the $2n\times2n$-matrix
\[\mat{A'}:=\begin{pmatrix}
&&&\pi_0\;\I_m\\&1&&&0\\&&&&&\pi_0\;\I_m\\\I_m\\&0&&&1\\&&\I_m
\end{pmatrix}\;.\]
To show that the image $\mat{A'}\;\F^{\perp}$ lies in the subspace $\pi^{-1}\F$ (which with respect to the basis of $\pi^{-1}\Lambda_m$ is given by the column span of the matrix $\F$), we give a square matrix $\mat{Q}$ of size $n$ that satisfies $\mat{A'}\;\F^{\perp}=\F\;\mat{Q}$:
\[\mat{Q}:=\begin{pmatrix}
&\pi_0\;\I_s\\\I_s\\-\iota({_{[m-s]}\mat{d}})&-\pi\;\iota({_{[m-s]}\mat{d}})&-\pi\;\I_{m-s}\\&&&1\\\iota(_{[m-s]}{^{[m]}\mat{d}})&\pi\;\iota(_{[m-s]}{^{[m]}\mat{d}})&&&-\pi\;\I_{m-s}
\end{pmatrix}\;.\]
That $\mat{Q}$ satisfies the above equation can be seen by straightforward---but tedious---matrix block-multiplications.


\subsection{Conclusion}
This completes the proof of the proposition: we have shown that the wedge local model contains an open subset that is isomorphic to affine space of dimension $rs$. Moreover, the open subset is a neighborhood of the special point $(\F_1,\G_1)$.
\end{proof}

Now the assertions of the theorem can be deduced.

\begin{proof}[Proof of Thm.~\ref{thm:reduced:reduced}]
We want to see that the local model contains an open subset that is isomorphic to affine space of dimension $rs$. For this, we show that the open subset constructed above is actually lying in the local model.

We consider the closed subscheme of the product of Grassmannians that consists of pairs satisfying the conditions of the wedge local model treated up to Rem.~\ref{rem:reduced:remainingconditions}:
\[Y:=\{(\F,\G)\in\Grass_{n,2n}\times\Grass_{n,2n};\ \text{conditions up to Rem.~\ref{rem:reduced:remainingconditions}}\}\;.\]
The standard open subset $\Grass_{n,2n}^J\times\Grass_{n,2n}^J$ of the product of Grassmannians is abbreviated to $U$. We have the following inclusions of closed subschemes:
\beqn\label{eq:reduced:MIloc}M_I^{\loc}\cap U\subset M_I^{\wedge}\cap U\subset Y\cap U\;.\eeqn
By Lem.~\ref{lem:localmodel:dimensiongenericfiber}, the generic fiber of the local model is irreducible and of dimension $rs$. As its closure (in the naive local model), the local model is irreducible as well. The structure morphism to $\OO_E$ is dominant, and since it is projective, the special fiber of the local model is nonempty. It follows from Chevalley's theorem \cite{EGA:IV:3}*{Thm.~13.1.3} that all irreducible components of the special fiber have dimension at least $rs$. By flatness, the special fiber of the local model is in fact equidimensional of dimension $rs$ \cite{L:2002}*{Prop.~4.4.16}.

We have seen in the proof of Prop.~\ref{prp:reduced:affine} that the $\OO_E$-scheme $Y\cap U$ is isomorphic to affine space of dimension $rs$; in particular, its special fiber and its generic fiber are both irreducible of dimension $rs$. Hence, on the level of reduced schemes, the inclusions in \eqref{eq:reduced:MIloc} are equalities. Since $Y\cap U$ is reduced, we obtain
\[Y\cap U=(Y\cap U)_{\red}=(M_I^{\loc}\cap U)_{\red}\subset M_I^{\loc}\cap U\;,\]
where the subscript ``red'' denotes the reduced structure. Together with \eqref{eq:reduced:MIloc}, this implies that the affine space $Y\cap U$ coincides with the open subset $M_I^{\loc}\cap U$ of the local model.
\end{proof}

\begin{rem}\label{rem:reduced:conditionssufficed}
The conditions verified after Rem.~\ref{rem:reduced:remainingconditions} are automatically satisfied: From the above, we obtain
the inclusion $Y\cap U\subset M_I^{\wedge}\cap U$, and the converse inclusion trivially holds true.
\end{rem}


\section{Irreducibility of the Special Fiber}
In this section, we will establish the second result required in the proof of the main theorem.

\begin{thm}[second step]\label{thm:irreducibility:irreducible}
Let $n=2m+1$ be odd and $I=\{m\}$. Then the special fiber of the local model $M_I^{\loc}$ is irreducible.
\end{thm}

\begin{proof}
The theorem is a consequence of an apparently weaker result, which is given in Prop.~\ref{prp:irreducibility:irreducibleneighborhood} below. Lemma~\ref{lem:irreducibility:worstpoint} ensures that this is actually sufficient.
\end{proof}


\subsection{Worst Point}\label{ssec:irreducibility:worstpoint}
Recall from sect.~\ref{ssec:reduced:bestpoint} that the special fiber of the local model is the union of Schubert varieties, enumerated by certain elements of the corresponding affine Weyl group. As in sect.~5.e of \cite{PR:2007}, we can see that there is a unique closed orbit, which has to be contained in the closed subset $\AAA^I(\mu)$. From sect.~2.d.2 of loc.cit., it follows that in the current situation the closed orbit consists of the single point $(\F_0,\G_0)$, given by the subspaces
\begin{gather*}
\F_0:=\Pi\;\Lambda_m\subset\Lambda_m\;,\\
\G_0:=\Pi\;\Lambda_{m+1}\subset\Lambda_{m+1}\;.
\end{gather*}
This point is, in some sense, at the opposite extreme of the previously considered best point $(\F_1,\G_1)$: it is contained in all irreducible components of the special fiber of the local model. Hence, we have the following result about the ``worst point'' (the naming is due to the occurrence of the ``worst singularities'' at this point):

\begin{lem}\label{lem:irreducibility:worstpoint}
To prove the irreducibility of the special fiber of the local model, it is sufficient to show that the worst point has an open neighborhood (in the special fiber of the local model) that is irreducible.\qed
\end{lem}

By the lemma, the next proposition is enough to complete the proof of Thm.~\ref{thm:irreducibility:irreducible}.

\begin{prp}\label{prp:irreducibility:irreducibleneighborhood}
Let $n=2m+1$ be odd and $I=\{m\}$. Then, in the special fiber of the local model, the point $(\F_0,\G_0)$ has an open neighborhood that is irreducible.
\end{prp}

\begin{proof}
We start with the description of an open neighborhood of the point $(\F_0,\G_0)$ in the special fiber of the wedge local model. In a later subsection, we consider the intersection with the local model and deduce the statement of the proposition.

As in the previous section, we use matrices to describe an open subset. We consider points of the special fiber; therefore, unless explicitly mentioned otherwise, all schemes in this section are over the residue field $k$. Since we want to prove an irreducibility result, it is enough to consider the reduced scheme structures; therefore, unless otherwise specified, all schemes are equipped with the reduced structure. Moreover, the schemes involved in this section are all of finite type over $k$. Hence, from the functorial point of view, it is enough to consider only geometric points, that is, $\bar{k}$-valued points, with $\bar{k}$ denoting a fixed algebraic closure of $k$ \cite{M:1999}*{\S6}.


\subsection{Conditions of the Wedge Local Model}
To simplify the upcoming calculations, we use rearranged bases of $\Lambda_m$ and $\Lambda_{m+1}$:
\begin{gather*}
\Lambda_m=\spn_{\OO_F}\{e_{m+2},\dots,e_n,\pi^{-1}e_1,\dots,\pi^{-1}e_m,e_{m+1}\}\;,\\
\Lambda_{m+1}=\spn_{\OO_F}\{e_{m+2},\dots,e_n,\pi^{-1}e_1,\dots,\pi^{-1}e_m,\pi^{-1}e_{m+1}\}\;.
\end{gather*}
As usual, we get corresponding $\OO_{F_0}$-bases by adding the $\pi$-multiples of the respective basis vectors above (in the prescribed order, cf.\ sect.\ \ref{ssec:reduced:wedgelocalmodel}).

Recall that the $\bar{k}$-valued points of the wedge local model are given by pairs of $\OO_F\ot\bar{k}$-subspaces $(\F,\G)$, with $\F\subset\Lambda_m\ot\bar{k}$ and $\G\subset\Lambda_{m+1}\ot\bar{k}$, subject to conditions (\hyperlink{hyp:reduced:N1}{N1})--(\hyperlink{hyp:reduced:N4}{N4}), (\hyperlink{hyp:reduced:W}{W}), and (\hyperlink{hyp:reduced:Pi}{Pi}). In particular, $\G$ is determined by $\F$ as its orthogonal complement, and it suffices to consider $\bar{k}$-valued points $\F$ of some standard open subset $\Grass_{n,2n}^J$ of the Grassmannian. In order for the corresponding open subset of the product of Grassmannians to contain the special point $(\F_0,\G_0)$, the index set $J$ has to correspond to the first $n$ elements of the above chosen $\OO_{\F_0}$-basis of $\Lambda_m$. Then the elements of $\Grass_{n,2n}^J(\bar{k})$ are represented by $2n\times n$-matrices
\beqn\label{eq:irreducibility:2nxn}\F=\twovec{\mat{X}}{\I_n}\;,\eeqn
with a square matrix $\mat{X}$ of size $n$ having entries in $\bar{k}$. With respect to the upcoming calculations, we subdivide $\mat{X}$ into four smaller blocks,
\[\mat{X}=\twomat{\mat{X_1}}{\mat{X_2}}{\mat{X_3}}{X_4}\;,\]
where $\mat{X_1}$ is a square matrix of size $n-1$ and $X_4$ a scalar (that is, a square matrix of size $1$).

We evaluate the conditions of the wedge local model. By construction, (\hyperlink{hyp:reduced:N1}{N1}) and (\hyperlink{hyp:reduced:N3}{N3}) are satisfied. The remaining conditions translate into constraints on the matrix $\mat{X}$.


\subsection{Lattice Inclusion Map}
Note that $\pi$ is zero in $\bar{k}$. The map induced by the inclusion $\Lambda_m\subset\Lambda_{m+1}$ is described by the matrix
\[\mat{\bar{A}}:=\twomat{\I_n-\mat{K}}{}{\mat{K}}{\I_n-\mat{K}}\;,\]
where the $n\times n$-matrix $\mat{K}$ is defined as
\[\mat{K}:=\twomat{\mat{0}_{2m}}{}{}{1}\;,\]
with $\mat{0}_{2m}$ denoting the zero matrix of size $2m$. We introduce the square matrix $\Jp_{2m}$ of size $2m+1$, following the notation of $\J_{2m}$:
\[\Jp_{2m}:=\twomat{\J_{2m}}{}{}{0}\;.\]
The natural perfect pairing \eqref{eq:reduced:perfpair} is then represented by the matrix
\[\mat{M'}:=\twomat{}{\Jp_{2m}-\mat{K}}{-\Jp_{2m}+\mat{K}}{}\;.\]
Condition (\hyperlink{hyp:reduced:N2}{N2}) requires that the map $\bar{A}$ restricts to a map $\F\to\F^{\perp}$. The image of $\F$ lies in the orthogonal complement of $\F$ if $\F^{\T}\;\mat{\bar{A}}^{\T}\;\mat{M'}\;\F=0$. We multiply these matrices; with
\[\mat{\bar{A}}^{\T}\;\mat{M'}=\twomat{\mat{K}}{\Jp_{2m}}{-\Jp_{2m}}{}\;,\]
we get the condition
$\mat{X}^{\T}\;\mat{K}\;\mat{X}+(\mat{X}^{\T}\;\Jp_{2m}-\Jp_{2m}\;\mat{X})=0$,
which in block form is given by
\beqn\label{eq:irreducibility:X1_symmetryMatrix}\twomat{{\mat{X_3}}^{\T}\;\mat{X_3}}{X_4\;{\mat{X_3}}^{\T}}{X_4\;\mat{X_3}}{{X_4}^2}+\twomat{{\mat{X_1}}^{\T}\;\J_{2m}-\J_{2m}\;\mat{X_1}}{-\J_{2m}\;\mat{X_2}}{{\mat{X_2}}^{\T}\;\J_{2m}}{0}=0\;.\eeqn
From the lower right blocks, we deduce ${X_4}^2=0$. Since $\bar{k}$ is a field, it follows that $X_4=0$. Then the upper right blocks give $\mat{X_2}=0$, and from the upper left blocks, the identity
${\mat{X_3}}^{\T}\;\mat{X_3}=\J_{2m}\;\mat{X_1}-{\mat{X_1}}^{\T}\;\J_{2m}$ follows.
By introducing an involution similar to $\iota$, the latter equation can be expressed more clearly. We multiply with $-\J_{2m}$ from the left and obtain
\beqn\label{eq:irreducibility:X1_symmetry}-\J_{2m}\;{\mat{X_3}}^{\T}\;\mat{X_3}=\mat{X_1}+\sigma(\mat{X_1})\;,\eeqn
with the involution $\sigma$ defined as follows: Let $\mat{B}$ be an arbitrary matrix with $k$ rows and $l$ columns. We write
\[\sigma(\mat{B}):=\mat{D_l}\;\mat{B}^{\T}\;\mat{D_k}\;,\]
where for an integer $i$ the matrix $\mat{D_i}$ is defined to be $\J_i$ if $i$ is even and $\Hh_i$ if $i$ is odd.
We calculate
\[\sigma(\mat{X_1})=\J_{2m}\;{\mat{X_1}}^{\T}\;\J_{2m}=\twomat{-\iota({_{[m]}{\mat{X_1}}_{[m]}})}{\iota({_{[m]}{\mat{X_1}}^{[m]}})}{\iota({^{[m]}{\mat{X_1}}_{[m]}})}{-\iota({^{[m]}{\mat{X_1}}^{[m]}})}\]
and see that $\sigma$ is a ``signed reflection'' at the antidiagonal. Therefore, \eqref{eq:irreducibility:X1_symmetry} is to some extent a symmetry condition.


\subsection{Pi-Stability}
Over $\bar{k}$ and with respect to the chosen bases, the action induced by multiplication with $\pi\ot1$ is given by the matrix $\mat{\bar{\Pi}}=\twomats{}{\mat{0}_n}{\I_n}{}$. Condition (\hyperlink{hyp:reduced:Pi}{Pi}) requires that $\F$ is $\bar{\Pi}$-stable. This holds true if there is an equation $\mat{\bar{\Pi}}\;\F=\F\;\mat{R}$, with a square matrix $\mat{R}$ of size $n$. We get
\[\twovec{\mat{0}_n}{\mat{X}}=\twovec{\mat{X}\;\mat{R}}{\mat{R}}\]
and deduce $\mat{R}=\mat{X}$ and $\mat{X}^2=0$. The latter equation is in block form given by
\[\twomat{{\mat{X_1}}^2}{\mat{0}}{\mat{X_3}\;\mat{X_1}}{0}=0\;,\]
from which we deduce the identities
\begin{gather}
\label{eq:irreducibility:X1_nilpotent}{\mat{X_1}}^2=0\;,\\
\label{eq:irreducibility:X3_zerodiv}\mat{X_3}\;\mat{X_1}=0\;.
\end{gather}
Here we have used that $\mat{X_2}$ and $\mat{X_4}$ are both zero.


\subsection{Wedge Condition}
Because the last column of $\mat{X}$ is identically zero, (\hyperlink{hyp:reduced:W}{W}) translates into a wedge condition for the $(2m+1)\times2m$-matrix composed of the blocks $\mat{X_1}$ and $\mat{X_3}$:
\beqn\label{eq:irreducibility:wedgeac}\wedge^{s+1}\twovec{\mat{X_1}}{\mat{X_3}}=0\eeqn
(recall that $s<r$, and $\pi=0\in\bar{k}$).


\subsection{Action of the Symplectic Group}
We are left with pairs of matrices $(\mat{X_1},\mat{X_3})$ subject to conditions \eqref{eq:irreducibility:X1_symmetry}--\eqref{eq:irreducibility:wedgeac}. We denote this space of matrices by $N$.

Recall the definition of the symplectic group of size $2m$: it is the group of invertible $2m\times2m$-matrices that preserve the antisymmetric form given by $\J_{2m}$,
\[\Sp_{2m}=\{\mat{g}\in\GL_{2m};\;\mat{g}^{\T}\;\J_{2m}\;\mat{g}=\J_{2m}\}\;.\]
This is a linear algebraic group, which we consider over $k$ and which acts on $N$ from the right:
\beqn\label{eq:irreducibility:groupaction}N\times\Sp_{2m}\to N\;,\quad((\mat{X_1},\mat{X_3}),\;\mat{g})\mapsto(\mat{g}^{-1}\;\mat{X_1}\;\mat{g},\;\mat{X_3}\;\mat{g})\;.\eeqn
We consider the projection morphism on the second factor,
\[\pr_{\mat{X_3}}:\ N\to\A^{2m}\;,\quad(\mat{X_1},\mat{X_3})\mapsto \mat{X_3}\;,\]
which is equivariant for the action of $\Sp_{2m}$ (with the action on $\A^{2m}$ given in the obvious way). By studying the fibers of $\pr_{\mat{X_3}}$, we expect a better understanding of the whole space $N$.

We write $\mat{c_0}:=\fourvex{1}{0}{\dots}{0}$ for the row vector of $\A^{2m}$ that has a one as first entry and zeros in the remaining $2m-1$ columns.

\begin{lem}
The orbit of $\mat{c_0}$ under the action of the symplectic group consists of all nonzero row vectors of $\A^{2m}$; that is, we have a surjection
\[\{\mat{c_0}\}\times\Sp_{2m}\twoheadrightarrow\A^{2m}\setminus\{\mat{0}\}\;,\quad(\mat{c_0},\mat{g})\mapsto \mat{c_0}\;\mat{g}\;.\]
\end{lem}

\begin{proof}
Let an arbitrary row vector $\mat{c_1}\neq\mat{0}\in\A^{2m}(\bar{k})$ be given. We will construct a symplectic matrix $\mat{g}\in\Sp_{2m}(\bar{k})$ that has $\mat{c_1}$ as first row. Then the lemma follows since $\mat{c_0}$ is obviously mapped to $\mat{c_1}$ under multiplication with $\mat{g}$ from the right.

Because the symplectic form is nondegenerate, a vector $\mat{d_1}\in\bar{k}^{2m}$ exists that pairs with $\mat{c_1}$ to a nonzero $\lambda_1\in\bar{k}$. We normalize $\mat{d_1}$ by scaling with ${\lambda_1}^{-1}$. Then $\mat{c_1}$ and $\mat{d_1}$ span a symplectic subspace $W_1\subset\bar{k}^{2m}$, with the symplectic form given by the matrix $\J_{2}$. The orthogonal complement ${W_1}^{\perp}$ is of dimension $2m-2$ and again symplectic. We repeat the above process by choosing an arbitrary $\mat{c_2}\in{W_1}^{\perp}\setminus\{\mat{0}\}$. After $m$ steps, we have obtained $2m$ vectors $\mat{c_1},\mat{d_1},\dots,\mat{c_m},\mat{d_m}$, which constitute a basis of $\bar{k}^{2m}$.

We define a square matrix $\mat{g}$ of size $2m$ with these basis vectors as columns (in a different order):
\[\mat{g}:=\begin{pmatrix}{\mat{c_1}}^{\T}&\dots&{\mat{c_m}}^{\T}&{\mat{d_m}}^{\T}&\dots&{\mat{d_1}}^{\T}\end{pmatrix}\;.\]
We have $\mat{g}^{\T}\;\J_{2m}\;\mat{g}=\J_{2m}$ by construction, and therefore, $\mat{g}$ is symplectic. Then $\mat{g}^{\T}$ is symplectic as well. This is true because the last equation is equivalent to $\mat{g}^{-1}=-\J_{2m}\;\mat{g}^{\T}\;\J_{2m}$, and by transposing both sides, we get the corresponding equation for $\mat{g}^{\T}$. Since $\mat{g}^{\T}$ has $\mat{c_1}$ as first row, the lemma is proven.
\end{proof}

Because of this transitivity result, there are essentially two fibers to examine: on the one hand, we have to look at the fiber over the zero vector, and on the other hand, we have to determine the fiber over $\mat{c_0}$.


\subsection{Zero Fiber}
The next lemma describes the fiber over the zero vector.

\begin{lem}\label{lem:irreducibility:zerofiber}
The fiber ${\pr_{\mat{X_3}}}\negthickspace^{-1}(\mat{0})$ is given by the $k$-scheme of $2m\times2m$-matrices $\mat{X_1}$ that satisfy the conditions
\[{\mat{X_1}}^2=0\;,\quad\mat{X_1}+\sigma(\mat{X_1})=0\;,\quad\wedge^{s+1}\mat{X_1}=0\;.\]
This scheme is irreducible. It has dimension $(2m-s)s$ if $s$ is even and dimension $(2m-s+1)(s-1)$ if $s$ is odd. In both cases, the dimension is smaller than $rs$.
\end{lem}

\begin{proof}
The description of the fiber is obvious from \eqref{eq:irreducibility:X1_symmetry}--\eqref{eq:irreducibility:wedgeac}; in particular, because $\mat{X_3}=0$, the wedge condition 
\eqref{eq:irreducibility:wedgeac} translates into the wedge condition involving $\mat{X_1}$ only.

The stabilizer of the zero vector is the whole symplectic group, $\Stab_{\mat{0}}=\Sp_{2m}$. It acts by conjugation on the elements $\mat{X_1}$ contained in the zero fiber. Pappas and Rapoport \cite{PR:2007} have considered this matrix scheme in sect.~5.e.\ of their paper. In the notation of loc.cit., it coincides with the special fiber of the matrix scheme $U_{r',s}^{\wedge}$, where we have set $r':=2m-s$. It is shown in loc.cit.\ that the special fiber is irreducible and of dimension $r's$ if $s$ is even and of dimension $(r'+1)(s-1)$ if $s$ is~odd. Since $r'=r-1$, the lemma is proven.

The argument of loc.cit.\ is as follows. We consider the matrix scheme $V_{r',s}^{\wedge}$ of $2m\times2m$-matrices $\mat{X_1}$ over $k$ that satisfy the conditions
\[{\mat{X_1}}^2=0\;,\quad\wedge^{s+1}\mat{X_1}=0\;.\]
This scheme is the union of the nilpotent $\GL_{2m}$-conjugation orbits $\OO_{2^i,1^{2m-i}}$ with \mbox{$i\le s$}, 
which respectively contain the Jordan matrices with exactly $i$ nilpotent Jordan blocks of size two and all other blocks being zero. The orbits have dimension $2(2m-i)i$, respectively, and the following closure relation holds true \cite{PR:2003}*{Rem.~4.2}:
\beqn\label{eq:irreducibility:closurerel}\OO_{2^i,1^{2m-i}}\subset\overline{\OO_{2^j,1^{2m-j}}}\ \text{if and only if}\ i\le j\;.\eeqn
We denote the fixed point scheme of $V_{r',s}^{\wedge}$ under the involution $-\sigma$ by $U_{r',s}^{\wedge}$. The symplectic group acts on this scheme 
by conjugation; slightly abusing notation, we denote the corresponding nilpotent conjugation orbits by the same symbols as above. It follows from Prop.~1 of \cite{O:1986} that $U_{r',s}^{\wedge}$ is the union of the orbits 
$\OO_{2^i,1^{2m-i}}$ with even $i\le s$. By Thm.~1 of loc.cit., a closure relation as in \eqref{eq:irreducibility:closurerel} also holds true in this context. We conclude that $U_{r',s}^{\wedge}$ is the closure of $\OO_{2^{s},1^{r'}}$ if $s$ is even and the closure of $\OO_{2^{s-1},1^{r'+1}}$ if $s$ is odd. The irreducibility of the symplectic group implies the irreducibility of its orbits and their closures. The dimension of these $\Sp_{2m}$-orbits is half the dimension of the corresponding $\GL_{2m}$-orbits \cite{KR:1971}*{Prop.~5 and its proof}.
\end{proof}


\subsection{Nonzero Fiber}
The following lemma gives a description of the fiber over $\mat{c_0}$.

\begin{lem}\label{lem:irreducibility:nonzerofiber}
The fiber ${\pr_{\mat{X_3}}}\negthickspace^{-1}(\mat{c_0})$ is given by the $k$-scheme $N'$ of pairs of matrices $(\mat{Y_1},\mat{Y_2})$ subject to the following conditions:
\[{\mat{Y_1}}^2=0\;,\quad \mat{Y_1}+\sigma(\mat{Y_1})=0\;,\quad\wedge^s\twovec{\mat{Y_1}}{\mat{Y_2}}=0\;,\quad \mat{Y_2}\;\mat{Y_1}=0\;.\]
Here $\mat{Y_1}$ denotes a square matrix of size $2m-2$ and $\mat{Y_2}$ a row vector of size $2m-2$.
\end{lem}

\begin{proof}
We describe the matrices $\mat{X_1}$ lying over $\mat{c_0}$ by evaluating \eqref{eq:irreducibility:X1_symmetry}--\eqref{eq:irreducibility:wedgeac}.

Equation \eqref{eq:irreducibility:X3_zerodiv} applied with $\mat{X_3}=\mat{c_0}$ implies that the first row of $\mat{X_1}$ is zero.

Since $-\J_{2m}\;{\mat{c_0}}^{\T}\;\mat{c_0}=\mat{K}\;\Hh_{2m}$, the left hand side of \eqref{eq:irreducibility:X1_symmetry} is the square matrix with all entries zero but the lower left, which is one. As noted before, $\sigma$ is a signed reflection at the antidiagonal; hence, \eqref{eq:irreducibility:X1_symmetry} implies that $\mat{X_1}$ has the following form:
\beqn\label{eq:irreducibility:X1}\mat{X_1}=\threemat{0}{\mat{0}}{0}{\sigma(\mat{Y_2})}{\mat{Y_1}}{\mat{0}}{1/2}{\mat{Y_2}}{0}\;.\eeqn
Here $\mat{Y_1}$ is a square matrix of size $2m-2$ that satisfies the symmetry condition
\beqn\label{eq:irreducibility:Y1_symmetry}\mat{Y_1}+\sigma(\mat{Y_1})=0\;,\eeqn
and $\mat{Y_2}$ is a row vector with $2m-2$ columns.

Because $\mat{c_0}$ has a unit in the first entry and zeros everywhere else, \eqref{eq:irreducibility:wedgeac} translates via Laplace expansion along $\mat{c_0}$ into a wedge condition for the $(2m-1)\times(2m-2)$-matrix composed of $\mat{Y_1}$ and $\mat{Y_2}$:
\beqn\label{eq:irreducibility:Y_wedge}\wedge^s\twovec{\mat{Y_1}}{\mat{Y_2}}=0\;.\eeqn

By \eqref{eq:irreducibility:X1_nilpotent}, the square of $\mat{X_1}$ has to be zero. Using \eqref{eq:irreducibility:X1}, this results in
\begin{gather}
\label{eq:irreducibility:Y1_nilpotent}{\mat{Y_1}}^2=0\;,\\
\label{eq:irreducibility:Y2_zerodiv}\mat{Y_2}\;\mat{Y_1}=0\;.
\end{gather}

As asserted, equations \eqref{eq:irreducibility:Y1_symmetry}--\eqref{eq:irreducibility:Y2_zerodiv} describe the fiber over $\mat{c_0}$.
\end{proof}

Next, we determine the stabilizer of $\mat{c_0}\in\A^{2m}$ and its action on the fiber over $\mat{c_0}$.

\begin{lem}
The stabilizer $\Stab_{\mat{c_0}}\subset\Sp_{2m}$ of $\mat{c_0}$ is given by symplectic matrices $\mat{g}$ of the following form:
\[\mat{g}=\threemat1{}{}{-\mat{g_1}\;\sigma(\mat{g_2})}{\mat{g_1}}{}{g_3}{\mat{g_2}}1\;,\]
with a symplectic matrix $\mat{g_1}$ of size $2m-2$, a row vector $\mat{g_2}$ of corresponding size, and a scalar $g_3$. Referring to these matrices by giving the essential data in the form of a triple $(\mat{g_1},\mat{g_2},g_3)$, the induced action on $N'$ can be described as follows:
\begin{gather}
N'\times\Stab_{\mat{c_0}}\to N'\;,\notag\\
\label{eq:irreducibility:actionN'}((\mat{Y_1},\mat{Y_2}),\;(\mat{g_1},\mat{g_2},g_3))\mapsto({\mat{g_1}}^{-1}\;\mat{Y_1}\;\mat{g_1},\;\mat{Y_2}\;\mat{g_1}-\mat{g_2}\;{\mat{g_1}}^{-1}\;\mat{Y_1}\;\mat{g_1})\;.
\end{gather}
\end{lem}

\begin{proof}
Let $\mat{g}\in\Sp_{2m}$ stabilize $\mat{c_0}$. Then the first row of $\mat{g}$ has to be $\mat{c_0}$. We subdivide $\mat{g}$ into blocks,
\[\mat{g}=\threemat{1}{\mat{0}}{0}{\mat{g_4}}{\mat{g_1}}{\mat{g_5}}{g_3}{\mat{g_2}}{g_6}\;,\]
with a square matrix $\mat{g_1}$ of size $2m-2$, a row vector $\mat{g_2}$ with $2m-2$ columns, and a scalar $g_3$.
We evaluate the condition of $\mat{g}$ being symplectic, $\mat{g}^{\T}\;\J_{2m}\;\mat{g}=\J_{2m}$. By multiplying the matrices on the left hand side and comparing the blocks of the matrix equation, we obtain the following constraints on the blocks of $\mat{g}$:
\begin{gather*}{\mat{g_1}}^{\T}\;\J_{2m-2}\;\mat{g_1}=\J_{2m-2}\;, \\
{\mat{g_1}}^{\T}\;\J_{2m-2}\;\mat{g_5}=0\;, \\
g_6+{\mat{g_4}}^{\T}\;\J_{2m-2}\;\mat{g_5}=1\;, \\
\mat{g_4}=-\mat{g_1}\;\J_{2m-2}\;{\mat{g_2}}^{\T}\;.
\end{gather*}
The first equation implies that $\mat{g_1}$ is symplectic; in particular, $\mat{g_1}$ is regular. Then $\mat{g_5}=0$ follows from the second equation, which in turn implies $g_6=1$ by the third equation. Together with the last equation, we obtain the description stated in the lemma. Note that $\mat{g}$ is determined by the triple $(\mat{g_1},\mat{g_2},g_3)$.

To see how the stabilizer acts on $N'$, we first determine the inverse of the stabilizer element $\mat{g}=(\mat{g_1},\mat{g_2},g_3)$. Since $\mat{g}$ is in particular symplectic, the inverse is given by $\mat{g}^{-1}=-\J_{2m}\;\mat{g}^{\T}\;\J_{2m}$. Multiplying the matrices on the right yields $\mat{g}^{-1}=(\mat{g_1}^{-1},-\mat{g_2}\;\mat{g_1}^{-1},-g_3)$.
Next, let an arbitrary element $(\mat{Y_1},\mat{Y_2})\in N'$ be given, with corresponding matrix $\mat{X_1}\in{\pr_{\mat{X_3}}}\negthickspace^{-1}(\mat{c_0})$. We calculate that the conjugate element $\mat{g}^{-1}\;\mat{X_1}\;\mat{g}$ corresponds to $({\mat{g_1}}^{-1}\;\mat{Y_1}\;\mat{g_1},\;\mat{Y_2}\;\mat{g_1}-\mat{g_2}\;{\mat{g_1}}^{-1}\;\mat{Y_1}\;\mat{g_1})\in N'$; therefore, the action of $\Stab_{c_0}$ on $N'$ is given in the asserted way.
\end{proof}

\begin{rem}
The entry $g_3$ of an element $(\mat{g_1},\mat{g_2},g_3)\in\Stab_{\mat{c_0}}$ does not occur on the right hand side of \eqref{eq:irreducibility:actionN'}; hence, it has no effect on the induced action on $N'$.
\end{rem}

\begin{rem}
The symplectic group of size $2m-2$ can be regarded as a subgroup of the stabilizer of $\mat{c_0}$: we have the inclusion morphism
\[\Sp_{2m-2}\hookrightarrow\Stab_{\mat{c_0}}\;,\quad \mat{g_1}\mapsto(\mat{g_1},0,0)\;.\]
The corresponding action on $N'$ is given by
\[N'\times\Sp_{2m-2}\to N'\;,\quad ((\mat{Y_1},\mat{Y_2}),\;\mat{g_1})\mapsto({\mat{g_1}}^{-1}\;\mat{Y_1}\;\mat{g_1},\;\mat{Y_2}\;\mat{g_1})\;,\]
which is completely analogous to \eqref{eq:irreducibility:groupaction}.
\end{rem}

Recall that we have set $r'=2m-s$. We consider the $k$-scheme $U_{r'-1,s-1}^{\wedge}$ defined analogously to the matrix scheme $U_{r',s}^{\wedge}$ from the proof of Lem.~\ref{lem:irreducibility:zerofiber}: it is given by square matrices $\mat{Y_1}$ of size $2m-2$ satisfying ${\mat{Y_1}}^2=0$, $\mat{Y_1}+\sigma(\mat{Y_1})=0$, and $\wedge^{s}\,\mat{Y_1}=0$.
The symplectic group acts on this scheme by conjugation and $U_{r'-1,s-1}^{\wedge}$ is the union of the finitely many $\Sp_{2m-2}$-orbits $\OO_{2^i,1^{2m-2-i}}$ with even $i\le s-1$. The orbits are irreducible, have dimension $(2m-2-i)i$, and a closure relation analogous to \eqref{eq:irreducibility:closurerel} holds true. Hence, there is an open dense orbit; it is $\OO_{2^{s-1},1^{r'-1}}$ if $s-1$ is even and $\OO_{2^{s-2},1^{r'}}$ if $s-1$ is odd. 

The first component $\mat{Y_1}$ of a point $(\mat{Y_1},\mat{Y_2})\in N'$ gives a point in $U_{r'-1,s-1}^{\wedge}$. This is true because \eqref{eq:irreducibility:Y_wedge} implies in particular $\wedge^s\,\mat{Y_1}=0$.
We study the projection morphism on the first factor,
\[\pr_{\mat{Y_1}}:\ N'\to U_{r'-1,s-1}^{\wedge}\;,\quad(\mat{Y_1},\mat{Y_2})\mapsto \mat{Y_1}\;,\]
which is equivariant for the action of $\Sp_{2m-2}$.

\begin{lem}
Over each orbit $\OO_{2^i,1^{2m-2-i}}$ with even $i\le s-1$, the projection morphism $\pr_{\mat{Y_1}}:\ N'\to U_{r'-1,s-1}^{\wedge}$ is a fibration into affine spaces. The inverse images of these orbits are irreducible subsets that partition $N'$. The inverse image of the open dense orbit has dimension $(2m-s)(s-1)$; the inverse images of the other orbits have smaller dimension.
\end{lem}

\begin{proof}
We fix an orbit $\OO_{2^i,1^{2m-2-i}}$ with even $i\le s-1$ and consider an arbitrary point $\mat{Y_1}$ thereof. We determine the points of $N'$ lying above $\mat{Y_1}$; that is, we identify the vectors $\mat{Y_2}$ giving elements $(\mat{Y_1},\mat{Y_2})\in N'$. The cases $i=s-1$ and $i<s-1$ are to be distinguished.

In the former case, the rank of the matrix $\mat{Y_1}$ equals $s-1$; thus, \eqref{eq:irreducibility:Y_wedge} implies that $\mat{Y_2}$ belongs to the image of $\mat{Y_1}$, and we can write $\mat{Y_2}=\mat{a}\;\mat{Y_1}$, with a row vector $\mat{a}$ of size $2m-2$. Then \eqref{eq:irreducibility:Y2_zerodiv}, which is the second condition mixing $\mat{Y_1}$ and $\mat{Y_2}$, automatically holds true: $\mat{Y_2}\;\mat{Y_1}=\mat{a}\;{\mat{Y_1}}^2=0$.
It follows that exactly the elements in the image of $\mat{Y_1}$, which is an $s-1$-dimensional vector space, correspond to points $(\mat{Y_1},\mat{Y_2})\in N'$. Locally on $\OO_{2^i,1^{2m-2-i}}$, this gives trivializations with linear isomorphisms as transition maps; in other words, we get a vector bundle over the orbit $\OO_{2^i,1^{2m-2-i}}$.

If $i<s-1$, \eqref{eq:irreducibility:Y_wedge} is automatically satisfied since the rank $i$ of $\mat{Y_1}$ is smaller than $s-1$. Hence, $\mat{Y_2}$ determines a point $(\mat{Y_1},\mat{Y_2})\in N'$ if and only if \eqref{eq:irreducibility:Y2_zerodiv} is satisfied, that is, if and only if $\mat{Y_2}$ lies in the kernel of $\mat{Y_1}$. It follows that every fiber is a vector space of dimension $2m-2-i$. Again, we get a vector bundle over the orbit $\OO_{2^i,1^{2m-2-i}}$.

The total space of a vector bundle over an irreducible base is irreducible, and its dimension is the sum of the base dimension and the typical fiber dimension. Hence, the dimension of the inverse image of the open dense orbit is calculated to be \mbox{$(2m-2-(s-1))(s-1)+(s-1)$} if $s-1$ is even and ${(2m-2-(s-2))(s-2)}+(2m-2-(s-2))$ if $s-1$ is odd. In both cases, this equals $(2m-s)(s-1)$. The inverse images of the other orbits (corresponding to even $i<s-2$) have smaller dimension $(2m-2-i)(i+1)$: note that $i+1<s-1\le m-1$.
\end{proof}

\begin{rem}
The action of $\Stab_{\mat{c_0}}$ on the inverse image of the open dense orbit is transitive if $s-1$ is even and not transitive if $s-1$ is odd. The actions on the inverse images of the other orbits are not transitive.
\end{rem}

Taking the respective closures in $N'$ of the inverse images of the orbits and omitting redundant terms yields the decomposition of the fiber over $\mat{c_0}$ into irreducible components:

\begin{cor}\label{cor:irreducibility:nonzerofiber}
The fiber ${\pr_{\mat{X_3}}}\negthickspace^{-1}(\mat{c_0})$ contains an irreducible component $Z_{\max}$ of dimension $(2m-s)(s-1)$. All other irreducible components, $Z_{\gamma}$ with $\gamma\in\Gamma$ (and $\Gamma$ a finite, possibly empty index set), have smaller dimension.\qed
\end{cor}


\subsection{Action of the Symplectic Group (continued)}
The action \eqref{eq:irreducibility:groupaction} of the symplectic group $\Sp_{2m}$ on $N$ gives rise to the surjective morphism
\begin{flalign*}
\phi:\ &{\pr_{\mat{X_3}}}\negthickspace^{-1}(\mat{c_0})\times\Sp_{2m}\to{\pr_{\mat{X_3}}}\negthickspace^{-1}(\mat{X_3}\neq\mat{0})\;,\\
&((\mat{X_1},\mat{c_0}),\;\mat{g})\mapsto(\mat{g}^{-1}\;\mat{X_1}\;\mat{g},\;\mat{c_0}\;\mat{g})\;.
\end{flalign*}
We consider the images under $\phi$ of the sets $Z_{\max}\times\Sp_{2m}$ and $Z_{\gamma}\times\Sp_{2m}$ with $\gamma\in\Gamma$: we denote the closures in $N$ by $Z'_{\max}$ and $Z'_{\gamma}$ with $\gamma\in\Gamma$, respectively.

\begin{lem}
The sets $Z'_{\max}$ and $Z'_{\gamma}$ with $\gamma\in\Gamma$ are irreducible subsets of $N$. The dimension of $Z'_{\max}$ equals $rs$. Any $Z'_{\gamma}$ with $\gamma\in\Gamma$ has smaller dimension.
\end{lem}

\begin{proof}
The irreducibility is obvious since images of irreducible subsets under morphisms are irreducible, and so are their closures. As for the dimension assertion, we consider the restriction of the projection morphism $\pr_{\mat{X_3}}$ to $\phi(Z_{\max}\times\Sp_{2m})$:
\[{\pr_{\mat{X_3}}}|\phi(Z_{\max}\times\Sp_{2m}):\ \phi(Z_{\max}\times\Sp_{2m})\to\A^{2m}\setminus\{\mat{0}\}\;,\quad(\mat{X_1},\mat{X_3})\mapsto\mat{X_3}\;.\]
This is a surjective morphism between irreducible schemes of finite type over $k$, with all fibers isomorphic to $Z_{\max}$. The base dimension and the typical fiber dimension sum up to the dimension of the total space \citelist{\cite{EGA:IV:3}*{Thm.~13.2.3} \cite{H:1977}*{Ex.~II.3.22}}.
Since $\phi(Z_{\max}\times\Sp_{2m})$ has the same dimension as its closure, 
we calculate $\dim{Z'_{\max}}=2m+(2m-s)(s-1)=rs$.
Analogous reasoning shows that the dimension of the other subsets is smaller.
\end{proof}

By Lem.~\ref{lem:irreducibility:zerofiber}, the subset ${\pr_{\mat{X_3}}}\negthickspace^{-1}(\mat{0})$ is irreducible of dimension smaller than $rs$. Together with $Z'_{\max}$ and $Z'_{\gamma}$ with $\gamma\in\Gamma$, we get a finite covering of $N$ by irreducible subsets. 
By omitting redundant terms, we obtain the decomposition of $N$ into irreducible components:

\begin{cor}\label{cor:irreducibility:decomposition}
The scheme $N$ contains the irreducible component $Z'_{\max}$, which has dimension $rs$. All other irreducible components of $N$ (if there are any at all) have smaller dimension.\qed
\end{cor}


\subsection{Intersection with the Local Model}
We will now pass to the local model, therewith finishing the proof of the proposition. The arguments resemble those from the proof of Thm.~\ref{thm:reduced:reduced}.

Recall that in the current section all schemes are over $k$ and equipped with the reduced structure (unless explicitly mentioned otherwise). The standard open subset $\Grass_{n,2n}^J\times\Grass_{n,2n}^J$ of the product of Grassmannians is abbreviated to $U$. As usual, $\bar{M}_I^{\loc}$ denotes the special fiber of the local model and $\bar{M}_I^{\wedge}$ the special fiber of the wedge local model. We have closed immersions 
\[{\bar{M}_I^{\loc}\cap U}\subset{\bar{M}_I^{\wedge}\cap U}\subset N\;.\]
Following the same arguments as given in the proof of Thm.~\ref{thm:reduced:reduced}, we deduce that the open subset ${\bar{M}_I^{\loc}\cap U}$ of the special fiber of the local model coincides with the irreducible component $Z'_{\max}$ of $N$: On the one hand, ${\bar{M}_I^{\loc}\cap U}$ is nonempty (it contains the special point $(\F_0,\G_0)$, see sect.~\ref{ssec:irreducibility:worstpoint}) and equidimensional of dimension~$rs$. On the other hand, by Cor.~\ref{cor:irreducibility:decomposition}, the decomposition of $N$ into irreducible components is given by $Z'_{\max}$, which has dimension $rs$, and irreducible components of smaller dimension (if there are any at all).

We conclude that $\bar{M}_I^{\loc}\cap U$ is an irreducible open neighborhood of the point $(\F_0,\G_0)$. This completes the proof of the proposition and, hence, also of Thm.~\ref{thm:irreducibility:irreducible}. With the ``second step'' established, the main theorem is finally proven.
\end{proof}

\begin{rem}\label{rem:irreducibility:oneextremeorbit}
In the case considered, the set $\AAA^I(\mu)$ (which was mentioned in sect.~\ref{ssec:reduced:bestpoint}) is the closure of a single extreme orbit and coincides with the geometric special fiber of the local model. This follows by dimension arguments in the same manner as above: note that the open subset constructed in sect.~\ref{sec:reduced} is a neighborhood of one of the best points and has dimension $rs$, and by the results of this section, the (geometric) special fiber of the local model is irreducible of dimension $rs$.
\end{rem}


\section{Other Special Parahoric Level Structures}
In the final section, we take a look at the cases treated by Pappas and Rapoport (see Rem.~\ref{rem:specialparahoric:Rapoport}). Transferring our methods from sect.~\ref{sec:reduced} to this situation, we obtain analogs of Thm.~\ref{thm:reduced:reduced} and Prop.~\ref{prp:reduced:affine}. In this way, we can strengthen some of Pappas and Rapoport's results.

\begin{thm}\label{thm:other:reduced}
Let $I=\{0\}$ if $n=2m+1$ is odd and $I=\{m\}$ if $n=2m$ is even. Then the local model $M_I^{\loc}$ contains an affine space of dimension $rs$ as open subset.
\end{thm}

\begin{proof}
Arguing as in the proof of Thm.~\ref{thm:reduced:reduced}, this is a consequence of the next proposition.
\end{proof}

\begin{prp}
Let $I=\{0\}$ if $n=2m+1$ is odd and $I=\{m\}$ if $n=2m$ is even. Then the wedge local model $M_I^{\wedge}$ contains an affine space of dimension $rs$ as open subset.
\end{prp}

\begin{proof}
In the next subsection, we handle the case $n=2m+1$ odd, $I=\{0\}$. In the subsection thereafter, the case $n=2m$ even, $I=\{m\}$ is dealt with.


\subsection{Odd Case}\label{ssec:other:odd}
Let $n=2m+1$ be odd and $I=\{0\}$. The essential part of the selfdual periodic lattice chain is given by
\[\ldots\rightarrow\Lambda_0\rightarrow\ldots\;,\]
with the standard lattice $\Lambda_0=\spn_{\OO_F}\{e_1,\dots,e_n\}$. Over $\OO_{F_0}$, we have the corresponding basis $e_1,\dots,e_n,\pi e_1,\dots,\pi e_n$. 

We examine $A$-valued points of $M_I^{\wedge}$, with $A$ an arbitrary $\OO_E$-algebra; that is, we consider $\OO_F\ot A$-submodules $\F\subset\Lambda_0\ot A$ subject to the conditions of the wedge local model. In particular, we deal with $A$-valued points of the Grassmannian $\Grass_{n,2n}$. Again, it is sufficient to consider the standard open subset $\Grass_{n,2n}^J$, where $J$ is the complement of the index set that corresponds to the basis elements $e_1,\dots,e_s,\pi e_1,\dots,\pi e_r$. The motivation for this choice of $J$ is the same as in sect.~\ref{sec:reduced}: we construct an open subset containing one of the best points; see Rem.~\ref{rem:other:bestpoint}.

The elements of $\Grass_{n,2n}^{J}$ can be described as the column span of $2n\times n$-matrices $\F$ as in \eqref{eq:reduced:F}. The remaining conditions of the wedge local model translate into:
\begin{enumerate}
\item[(\hypertarget{hyp:other:N3}{N3})] $\F=\F^\perp$, with $\F^\perp$ denoting the orthogonal complement of $\F$ under the natural perfect pairing
\[\langle\mbox{ , }\rangle\ot A:\ (\Lambda_0\ot A)\times(\Lambda_0\ot A)\to A\;.\]
With respect to the chosen basis, the form is represented by the antisymmetric matrix $-\J_{2n}$.
\item[(\hypertarget{hyp:other:N4}{N4})] The characteristic polynomial of $\Pi|\F$ is given by
\[\det(T\;\id-\Pi|\F)=(T-\pi)^s(T+\pi)^r\in A[T]\;.\]
\item[(\hypertarget{hyp:other:W}{W})] We have
\begin{gather*}
\wedge^{r+1}(\Pi-\sqrt{\pi_0}|\F)=0\;,\\
\wedge^{s+1}(\Pi+\sqrt{\pi_0}|\F)=0\;.
\end{gather*}
\item[(\hypertarget{hyp:other:Pi}{Pi})] $\F$ is $\Pi$-stable.
\end{enumerate}

The parts of the proof of Prop.~\ref{prp:reduced:affine} concerning (\hyperlink{hyp:other:N4}{N4}), (\hyperlink{hyp:other:W}{W}), and (\hyperlink{hyp:other:Pi}{Pi}) are identically applicable to the current case, yielding the same identities \eqref{eq:reduced:PiStable}--\eqref{eq:reduced:a_s} for the variables $a$, $b$, $c$, and $d$ of the subspaces $\F$.
In particular, $b$ is determined in terms of the other variables by \eqref{eq:reduced:b}.

Contrary to the previous case, the condition concerning the restrictions of the lattice inclusion maps is trivial this time. Instead, (\hyperlink{hyp:other:N3}{N3}) gives further constraints on the $a$-, $c$-, and $d$-variables. It is enough to show that $\F\subset\F^\perp$. This translates into three subconditions for the columns of $\F$:
\begin{enumerate}
\item[(\hypertarget{hyp:other:N3a}{N3.a})] The first $s$ columns are perpendicular to each other.
\item[(\hypertarget{hyp:other:N3b}{N3.b})] The first $s$ columns are perpendicular to the last $r$ columns.
\item[(\hypertarget{hyp:other:N3c}{N3.c})] The last $r$ columns are perpendicular to each other.
\end{enumerate}

The first condition gives
\beqn\label{eq:other:c}\mat{c}=\iota(\mat{c})\;,\eeqn
which means that $c$ has to be symmetric with respect to the antidiagonal. The solution space of this system of linear equations has dimension ${s(s+1)}/{2}$.

Condition (\hyperlink{hyp:other:N3b}{N3.b}) reads $\Hh_s\;\mat{d}+\mat{a}^{\T}\;\Hh_r=0$. We rearrange and multiply with $\Hh_s$ from the left to get 
\beqn\label{eq:other:d}\mat{d}=-\iota(\mat{a})\;.\eeqn
Hence, the $d$-variables are determined by the $a$-variables. We split the last equation into $^{[s]}\mat{d}=-\iota(\mat{a}_{[s]})$ and $_{[r-s]}\mat{d}=-\iota(\mat{a}^{[r-s]})$; with these identities, the corresponding terms in \eqref{eq:reduced:a_s} are substituted. 
Rearranging yields
\beqn\label{eq:other:maineq}(\id+\iota)(\mat{a}_{[s]})=-\iota(\mat{a}^{[r-s]})\;\mat{a}^{[r-s]}\;.\eeqn
This is analogous to the situation at the end of sect.~\ref{ssec:reduced:latticeinclusion}:
both sides of \eqref{eq:other:maineq} 
are symmetric with respect to the antidiagonal, and we can express the $a$-variables on or below the first angle bisector in terms of those above. Thus, there remain $s(r-s)+{s(s-1)}/2$ free $a$-variables. Taken together with the $c$-variables, which are independent of the $a$-variables, we end up with an affine space of the desired dimension
\[\frac{s(s+1)}2+s(r-s)+\frac{s(s-1)}2=rs\;,\]
provided (\hyperlink{hyp:other:N3c}{N3.c}) is redundant.

We have to show that $\mat{b}^{\T}\;\Hh_r-\Hh_r\;\mat{b}=0$, or equivalently, that $\iota(\mat{b})=\mat{b}$. This holds true indeed, as follows from the description \eqref{eq:reduced:b} of $\mat{b}$ together with \eqref{eq:other:c} and \eqref{eq:other:d}. This proves the proposition in the case $n=2m+1$ odd.

\begin{rem}\label{rem:other:bestpoint}
The point given by the subspace $\F_1:=\spn_k\{e_1,\dots,e_s,\pi e_1,\dots,\pi e_r\}$ lies in the special fiber of the open subset constructed above. As in Rem.~\ref{rem:reduced:bestpoint}, it follows that this is one of the special points mentioned in sect.~\ref{ssec:reduced:bestpoint}.
\end{rem}


\subsection{Even Case}
Let $n=2m$ be even and $I=\{m\}$. Then the selfdual periodic lattice chain is given by
\[\ldots\rightarrow\Lambda_m\rightarrow\ldots\;,\]
with $\Lambda_m=\spn_{\OO_F}\{f_1,\dots,f_n\}$ denoting the standard lattice, where we have set $f_1:=\pi^{-1}e_1,\dots,f_m:=\pi^{-1}e_m,f_{m+1}:=e_{m+1},\dots,f_n:=e_n$. As usual, by adding the $\pi$-multiples of the respective basis vectors, we get a basis over $\OO_{F_0}$.

We proceed as in the odd case and construct an open subset of the wedge local model by considering $A$-valued points of $\Grass_{n,2n}^J$, with the complement of $J$ corresponding to the basis elements $f_1,\dots,f_s,\pi f_1,\dots,\pi f_r$. The points are represented as column spans $\F$ as in \eqref{eq:reduced:F}, and are subject to (\hyperlink{hyp:other:N3}{N3}), (\hyperlink{hyp:other:N4}{N4}), (\hyperlink{hyp:other:W}{W}), and (\hyperlink{hyp:other:Pi}{Pi}) from the previous subsection~\ref{ssec:other:odd}. Of course, this time the orthogonality condition has to be with respect to the natural perfect pairing $(\mbox{ , })\ot A:\ (\Lambda_m\ot A)\times(\Lambda_m\ot A)\to A$; the form is represented by the symmetric matrix
\[M'':=\twomat{}{-\J_{2m}}{\J_{2m}}{}\;.\]
The wedge condition is only posed if $r\neq s$.

As before, (\hyperlink{hyp:other:N4}{N4}), (\hyperlink{hyp:other:W}{W}), and (\hyperlink{hyp:other:Pi}{Pi}) yield the identities \eqref{eq:reduced:PiStable}--\eqref{eq:reduced:a_s}. Note that \eqref{eq:reduced:s_b_r-s} and \eqref{eq:reduced:r-s_b_r-s}, which were obtained using the wedge condition, trivially hold true if $r=s$.
Equation \eqref{eq:reduced:b} determines $b$ in terms of the other variables.

We evaluate the orthogonality condition (\hyperlink{hyp:other:N3}{N3}), or equivalently, the three subconditions (\hyperlink{hyp:other:N3a}{N3.a}), (\hyperlink{hyp:other:N3b}{N3.b}), and (\hyperlink{hyp:other:N3c}{N3.c}).

Recall that $s\le m$. The first condition gives
\beqn\label{eq:other:c2}\mat{c}=-\iota(\mat{c})\;.\eeqn
That is, $c$ has to be antisymmetric with respect to the antidiagonal. The solution space of this system of linear equations has dimension ${s(s-1)}/2$.

Condition (\hyperlink{hyp:other:N3b}{N3.b}) gives
\begin{align*}
0&={^{[s]}\!\F}^{\;\T}\;M''\;{_{[r]}\F}\\
&=\fourvex{\I_s}{\mat{a}^{\T}}{\mat{0}}{\mat{c}^{\T}}\;\twomat{}{-\J_{2m}}{\J_{2m}}{}\;\fourvec{\mat{0}}{\mat{b}}{\I_r}{\mat{d}}\\
&=\fourvex{\I_s}{\mat{a}^{\T}}{\mat{0}}{\mat{c}^{\T}}\;\fourvec{-\Hh_s\;\mat{d}}{-\J_{m,m-s}}{\J_{m-s,m}\;\mat{b}}{\mat{0}}\\
&=-\Hh_s\;\mat{d}-\mat{a}^{\T}\;\J_{m,m-s}\;,
\end{align*}
which is equivalent to 
\beqn\label{eq:other:d2}\mat{d}=-\Hh_s\;\mat{a}^{\T}\;\J_{m,m-s}\;.\eeqn
Consequently, all $d$-variables are determined by the $a$-variables. The first $s$ columns of the matrix equation \eqref{eq:other:d2} give $^{[s]}\mat{d}=\iota(\mat{a}_{[s]})$, and the last $r-s$ columns give $_{[r-s]}\mat{d}=-\Hh_s\;{\mat{a}^{[r-s]}}^{\T}\;\J_{2(m-s)}$;
these identities are used to substitute the corresponding terms in \eqref{eq:reduced:a_s}. We obtain
\beqn\label{eq:other:maineq2}(-\id+\iota)(\mat{a}_{[s]})=\Hh_s\;{\mat{a}^{[r-s]}}^{\T}\;\J_{2(m-s)}\;\mat{a}^{[r-s]}\;.\eeqn
Both sides of \eqref{eq:other:maineq2} are antisymmetric with respect to reflection at the antidiagonal: for the left hand side this is obvious; for the right hand side, which is temporarily denoted by $\mat{B}$, we calculate
$\iota(\mat{B})=\Hh_s\;{\mat{a}^{[r-s]}}^{\T}\;(-\J_{2(m-s)})\;\mat{a}^{[r-s]}\;\Hh_s\;\Hh_s=-\mat{B}$. In analogy to the previous cases, we get ${s(s-1)}/2$ equations, which express the $a$-variables below the first angle bisector in terms of those on or above. Therefore, the number of free $a$-variables is $(r-s)s+{s(s+1)}/2$. 

Since the $c$-variables are independent of the $a$-variables, we end up with an affine space of the asserted dimension
\[\frac{s(s-1)}2+(r-s)s+\frac{s(s+1)}2=rs\;,\]
provided we can show that the last $r$ columns of $\F$ are now automatically perpendicular to each other.

The remaining condition translates into $\mat{b}=-\J_{m,m-s}\;\mat{b}^{\T}\;\J_{m,m-s}$. Taking into account the block form of $\mat{b}$ described in \eqref{eq:reduced:b}, we split the last equation into four parts corresponding to the respective blocks of $\mat{b}$:
\begin{gather*}
\pi\;\mat{a}^{[r-s]}=\J_{2(m-s)}\;(-\pi\;_{[r-s]}\mat{d})^{\T}\;\Hh_s\;,\\
-\pi\;\I_{r-s}=-\J_{2(m-s)}\;(-\pi\;\I_{r-s})^{\T}\;\J_{2(m-s)}\;,\\
\pi_0\;\mat{c}+\pi\;_{[r-s]}\mat{d}\;\mat{a}^{[r-s]}=-\Hh_s\;(\pi_0\;\mat{c}+\pi\;_{[r-s]}\mat{d}\;\mat{a}^{[r-s]})^{\T}\;\Hh_s\;,\\
-\pi\;_{[r-s]}\mat{d}=\Hh_s\;(\pi\;{\mat{a}^{[r-s]}})^{\T}\;\J_{2(m-s)}\;.
\end{gather*}
All these equations hold true, as can be easily verified using \eqref{eq:other:c2} and \eqref{eq:other:d2}. This proves the proposition in the case $n=2m$ even.
\end{proof}

\begin{rem}\label{rem:other:bestpointeven}
The point given by the subspace $\F_1:=\spn_k\{f_1,\dots,f_s,\pi f_1,\dots,\pi f_r\}$ lies in 
the special fiber of
the open subset constructed above. As in Rem.~\ref{rem:reduced:bestpoint}, it follows that this is one of the special points mentioned in sect.~\ref{ssec:reduced:bestpoint}.
\end{rem}





\end{document}